\documentclass[12pt,reqno]{amsart}
\usepackage[margin=1in]{geometry}
\makeatletter
\def\section{\@startsection{section}{1}%
	\z@{.7\linespacing\@plus\linespacing}{.5\linespacing}%
	{\bfseries\normalfont\scshape
		\centering
}}
\def\@secnumfont{\bfseries}
\makeatother
\usepackage{hyperref}
\usepackage{amsmath,amsfonts,amsthm,txfonts}
\usepackage{dsfont}
\usepackage{enumerate}
\usepackage{upgreek} 
\usepackage{xcolor}
\usepackage{cases}
\newtheorem{theorem}{Theorem}[section]

\newtheorem{lemma}[theorem]{Lemma}
\newtheorem{proposition}{Proposition}

\theoremstyle{definition}
\newtheorem{definition}[theorem]{Definition}
\newtheorem{remark}{Remark}
\newtheorem{Ass}{Assumption}

\begin{document}
	\title[Identification of thermal expansion coefficient]{Identification of thermal expansion coefficient in \vspace{.05in}\\ a thermoelastic plate from final time measured \vspace{.05in}\\ displacement} 
	\author{ A. Dileep}\address{Department of Mathematical Sciences, Hanbat National University, Daejeon, Republic of Korea, 34158.}
    \email{dileepanjuna10@gmail.com}
	\author{S. Patnaik}
    \address{Johann Radon Institute for Computational and Applied Mathematics (RICAM), Austrian Academy of Sciences, Altenberger Strasse 69, 4040 Linz, Austria}
    \email{sidhartha.patnaik@ricam.oeaw.ac.at}
  	\author{K. Sakthivel}
    \address{Department of Mathematics, Indian Institute of Space Science and Technology, Trivandrum- 695 547, INDIA.}
	\email{sakthivel@iist.ac.in}
    \author{A. Hasanov}
    \address{Department of Mathematics, Kocaeli university, 41001, Kocaeli, Turkey.}
    \email{alemdar.hasanoglu@gmail.com}
	\begin{abstract}
		  We investigate a coupled thermoelastic plate system consisting of a fourth-order displacement equation and a heat evolution  equation linked through a spatially varying coupling factor $\alpha(x)$. The model accounts for thermoelastic interactions through the operators $\operatorname{div}(\alpha(x)\nabla \theta)$ and $\operatorname{div}(\alpha(x)\nabla u_t)$. We establish the well-posedness of the direct problem under homogeneous Neumann conditions for $u$ and Dirichlet conditions for $\theta$, deriving optimal energy estimates and demonstrating continuous dependence of solutions on the given data.
		We further introduce an input-output operator corresponding to the considered inverse problem and show that it is compact and Lipschitz continuous, confirming the ill-posed nature of the associated inverse problem. Using these properties, the inverse problem is formulated as a minimization problem for the Tikhonov functional, and we establish the existence of a minimizer.  
	\end{abstract}
	
	\subjclass{35R30, 74F05, 35B65.}
	\keywords{Kirchhoff-Love plate, heat equation, thermoelastic plate equation, inverse coefficient problem, final time output.}
	
	\maketitle
	
		\section{Introduction}
	In this paper, the inverse problem of determining the unknown thermal expansion coefficient in an anisotropic thermoelastic plate is studied. The vibration of the plate is governed by the Kirchhoff-Love plate equation, and the temperature distribution is described by the heat equation based on Fourier’s law of heat conduction. In this coupled system, temperature gradients generate internal stresses while deformation rates feed back into the thermal field, producing a bidirectional coupling that strongly influences structural behavior. In engineering applications, structural components are often subjected to non-uniform thermal gradients that vary across different directions due to material anisotropy. This nonuniformity can lead to deformation, changes in geometry from thermal expansion, and a reduction in the material strength and stiffness. To accurately capture the plate’s response under the above conditions, the study of thermoelasticity in anisotropic material is essential.
	
	The thermoelastic behavior of anisotropic plates is fundamental in applications such as aerospace panels, MEMS devices, and composite structural components, where thermal loads and mechanical vibrations interact in a nontrivial manner. In these applications it is crucial for understanding how thermal and mechanical interactions influence the dynamic behavior and stability of plates under external loads (see, \cite{andrew2018}, \cite{bhullar}, \cite{timoshenko}).
	
	In this study, we consider a thermoelastic plate whose mechanical and thermal responses are governed by a fully coupled system of equations that captures both the anisotropic material behavior and the spatial variation of the thermoelastic coupling coefficient. This modeling framework enables us to account for direction-dependent thermal conduction, non-uniform mechanical stiffness, and the bidirectional influence between deformation and temperature evolution. In particular, the unknown coefficient $\alpha(x)$ enters directly into the coupling terms, controlling the strength with which temperature gradients induce bending forces and the rate of deformation contributes to heat generation. 
	
	If we consider a bounded domain $\Omega \subset \mathbb{R}^2$ with smooth boundary $\Gamma$, and define  $\Omega_T:=\Omega \times (0,T]$ along with $\Gamma_T:=\Gamma \times [0,T]$, then the evolution of the plate's transverse displacement $u(x, t)$, and temperature deviation $\theta(x, t)$ is governed by the following system of coupled non-homogeneous, thermally anisotropic \emph{thermoelastic plate} equations (see, \cite{meller2001}, \cite{JElagnese} and \cite{J.E. Lagnese}):	
	\begin{subnumcases}{\label{1aa}}
		u_{tt} - r\,\Delta u_{tt} + \Delta^{2}u + \operatorname{div}(\alpha(x)\nabla\theta)
		= F(x,t), \quad\qquad   (x,t)\in \Omega_T, \label{1a} \\[4pt]
		\theta_{t} - \Delta\theta - \operatorname{div}(\alpha(x)\nabla u_{t})
		= G(x,t), \qquad \qquad \quad\ \quad(x,t)\in \Omega_T, \label{1b}  \\[0pt]
		u = \frac{\partial u}{\partial\nu} = 0,\qquad 
		\theta = 0, \qquad\qquad\qquad\,\qquad\qquad (x,t)\in \Gamma_T, \label{1c} \\[2pt]
		u(x,0) = u_{0},\quad 
		u_{t}(x,0) = v_{0},\quad 
		\theta(x,0) = \theta_{0}, \qquad\quad\, x\in\Omega. \label{1d}
	\end{subnumcases}
	We now define the physical significance of the parameters and functions appearing in the coupled system \eqref{1aa}. The term $-r\Delta u_{tt}$ in the mechanical equation accounts for the rotational inertia of the plate filaments, with the constant $r > 0$ being proportional to the square of the plate thickness. The function $\alpha(x)$ is the spatially varying thermal expansion coefficient, which acts as the coupling mechanism. This coupling is two-way: the term $\operatorname{div}(\alpha(x)\nabla \theta)$ represents the bending force induced by the thermal gradient ($\nabla \theta$), and the term $-\operatorname{div}(\alpha(x)\nabla u_t)$ represents the heat generated by the rate of mechanical deformation ($u_t$). Finally, the boundary conditions reflect a clamped plate with a constant (zero) temperature maintained along the edge: $u=\frac{\partial u}{\partial\nu}=0$ and $\theta=0$ on $\Gamma_T$.
	
%This system captures the full two-way thermoelastic coupling, where thermal gradients ($\nabla \theta$) induce mechanical forces, and the rate of mechanical deformation ($u_t$) acts as a heat source.

	%Although there have been tremendous advancements in the domains such as nuclear engineering, aircraft, machine construction, and several other related fields over the past few decades, the coupled effect between deformation and temperature has been a key factor for the solution of many thermal shock problems. The analysis of the impact of the temperature field and stresses produced during thermal shock, which might result in premature failure, and understanding of thermoelastic damping in devices is crucial in many engineering situations (see, \cite{andrew2018}, \cite{bhullar},  \cite{timoshenko}). Therefore,  it is extremely important to understand the effect of the thermal and mechanical sources that result in the material's deflection. 
	%	Given all the parameters, initial data, and source terms of \eqref{1}, finding the solution $(u,\theta)$ is referred to as the \emph{direct problem}. 
	In this paper, we determine the unknown thermal expansion coefficient $ \alpha(x)$ in the thermoelastic plate equation \eqref{1a}-\eqref{1d} from the final time-measured displacement
	\begin{eqnarray}\label{21}
		u_{T}(x):=u(x,T),\ \ x \in \Omega.\label{i1}
	\end{eqnarray}
The unknown coefficient $\alpha(x)$ is appears in both equations \eqref{1a}, \eqref{1b} and, more importantly, it is embedded within the differential operator. Consequently, the inverse coefficient problem is considerably more challenging than the corresponding inverse problems involving a single equation.

The determination of the thermal expansion coefficient in a thermoelastic plate is essential for various applications where structural performance is affected by the temperature changes. Thermoelastic plates are used in satellite panels and aircraft components to maintain stability or functionality in the case of temperature fluctuations. In electronics, printed circuit boards and microelectronic mechanical systems depend on thermoelastic materials to prevent deformation and ensure long-term reliability under thermal stress. Moreover, manufacturing processes like welding require accurate thermal expansion data to minimize residual stresses and ensure dimensional accuracy.
	Because $\alpha(x)$ may vary spatially in composites or in materials subjected to thermal aging, reconstructing this coefficient is practically important for diagnostics, material characterization, and integrity assessment.
	These applications lead to the importance of identifying the thermal expansion coefficient to optimize structural integrity and performance across industries (see, \cite{esuhir},\cite{xxing}).
    
    Next, we briefly review the literature on inverse problems for classical thermoelastic system, which involves a coupled system of hyperbolic equation for displacement, and the heat equation for temperature. The identification of the coefficients in such systems, based on measured displacement in a subdomain over a sufficiently large time interval, was studied in \cite{wu1}. Using the Carleman estimate method, the authors established a Lipschitz stability estimate for the solution to the inverse problem. The inverse source problem for a generalized thermoelastic system was analyzed using the Carleman estimate method in \cite{bellassoued2010}. In \cite{bwu}, the determination of a spatially varying unknown source term in a thermoelastic system with memory from measured displacement in a subdomain over a sufficiently large time interval was examined. The authors provided a Hölder stability estimate using the Carleman estimate method. The inverse problem of a space-dependent vector source in a thermoelastic system, based on measured final-time deflection, was investigated in \cite{van2015}. The authors proved the uniqueness of the solution to the direct problem using a variational approach and presented a numerical reconstruction for the solution to the inverse problem.
	
	Recently, the first study on the simultaneous identification of spatial load $ F(x,t)$ and temporal load $G(x,t)$ in the following thermoelastic plate
	\begin{subnumcases}{}
		& $\rho h  u_{tt}-\omega \Delta u_{t}+D\Delta ^2 u+ \beta_1 \Delta \theta = F(x,t), \;\; \; (x,t)\in \Omega_T,$ \nonumber   \\ [1pt]
		&$\beta_2	\theta_{t}- \Delta\theta -\beta_{3} \Delta u_t+\beta_4 \theta= G(x,t), \quad\; \ (x,t)\in \Omega_T,$\nonumber
	\end{subnumcases}
	with same set of boundary and initial conditions in \eqref{1c}-\eqref{1d}, using the measured deflection at the final time was conducted in \cite{anjuna}. They proved the stability estimate and uniqueness result for the solution to the inverse problem.  Motivated by this result, we aim to identify the variable thermal coefficient in the anisotropic plate model \eqref{1a}-\eqref{1d}. However, unlike \cite{anjuna}, where the unknowns appear as additive source terms, in this work, the unknown coefficient multiplies gradient fields inside the operator, leading to a fundamentally nonlinear inverse problem. Further, in the model \eqref{1a}-\eqref{1d}, we also include the rotational effect ($r\Delta u_{tt}$) in the plate, which also brings substantial technicalities to this study.  
	
	In this paper, we establish the existence of a quasi-solution to the coefficient inverse problem using the approach introduced in \cite{Tikhonov} (see also \cite{hasanouglu2017introduction}). This method has been extensively applied to study inverse source problems in beam and plate models, diffusion equations and higher order systems of PDEs (see, for example, \cite{OAA, sdbt2014, hasanov2007, hasanov2009identification, hasanouglu2019, van2019,NKM,sakthivel2011,TLK}).
	
	To the best of our knowledge, this is the first work addressing the identification of a spatially varying thermoelastic coupling coefficient in a fully coupled anisotropic Kirchhoff–Love plate model using only a final-time displacement measurement.
	
	The main contributions of this paper are summarized as follows:
	\begin{enumerate}
		\item By adopting the method developed in \cite{anjuna}, we establish the existence and uniqueness of the solution to the direct problem with minimal regularity assumptions on the thermal expansion coefficient $ \alpha(x).$ However, unlike the isotropic plate, the anisotropic plate poses significant challenges in proving the existence and uniqueness of the solution to the direct problem \eqref{1a}-\eqref{1d}. We explicitly handle these analytical obstacles by establishing appropriate energy estimates and compactness arguments tailored to the anisotropic coupling structure.
		
		\item The nonlinear inverse problem of identifying the thermal expansion coefficient $ \alpha(x)$ from the final-time measured deflection $u_T(x)=u(x,T)$ is solved using the Tikhonov regularization method (\cite{Tikhonov},\cite{hasanouglu2017introduction}). 
		
		The inverse problem of determining the coefficients embedded within the differential operators in both equations of the coupled thermoelastic plate system presents a significant challenge. This is distinct from the source identification problem as well as the coefficient identification problem for a single equation. Nevertheless, the results presented in this article offer a strong foundation for further analysis of various coefficients within this specific model.

			% combine and modify this

	\end{enumerate}

	The other sections of this paper are organized as follows. Section \ref{S-WP} provides a detailed analysis of the well-posedness of the direct problem \eqref{1a}-\eqref{1d}. The regularity of the weak solution to the direct problem is studied in Section \ref{S-RWS}. In Section \ref{S-IP}, the inverse problem is formulated, and the compactness as well as Lipschitz continuity of the input-output operator is given. The existence of a minimizer for the functional is studied in this section.

	\section{Well-posedness of the direct problem}\label{S-WP}
	In this section, we analyze the existence, uniqueness, and regularity properties of the solutions to the direct problem \eqref{1a}-\eqref{1d}, which are essential for the analysis of the inverse problem.

	\subsection{Function spaces and classical inequalities}
	First, to begin the construction of a weak solution, we recall the relevant function spaces, norms, and basic inequalities that will be used throughout the analysis. We introduce the following Sobolev spaces equipped with boundary conditions compatible with our model.
	\iffalse
    \begin{eqnarray*} 
		H_0^1(\Omega)&:=&\{ w\in H^1(\Omega): w=0 \ \mbox{on} \ \Gamma \}, \\
		H_0^2(\Omega)&:=& \{ v\in H^2(\Omega): v=0 \ \mbox{and} \ \frac{\partial v}{\partial \textnormal{n}} =0 \ \mbox{on}\ \Gamma \},\\
		\mathcal{V}^2(\Omega)&:=&\{ w\in H^2(\Omega): w=0 \ \mbox{on} \ \Gamma \},\\
		\mathcal{V}^3(\Omega)&:=& \{ v\in H^3(\Omega): v=0 \ \mbox{and} \ \frac{\partial v}{\partial \textnormal{n}} =0 \ \mbox{on}\ \Gamma \}  .
	\end{eqnarray*}
	Each of the spaces $H_0^1(\Omega)$, $H_0^2(\Omega)$, $\mathcal{V}^2(\Omega)$ and $\mathcal{V}^3(\Omega)$ is endowed with the standard Sobolev norms inherited from $H^1(\Omega)$, $H^2(\Omega)$, $H^2(\Omega)$ and $H^3(\Omega)$  respectively. 
\fi
For an integer $k \ge 1$ and $j \in \{0,1\}$, define
\begin{equation*}
\mathcal{X}^k_j(\Omega) := \left\{ w \in H^k(\Omega) : w = 0 \text{ on } \Gamma, \ \text{and } \frac{\partial w}{\partial \nu} = 0 \text{ on } \Gamma \text{ if } j = 1 \right\},
\end{equation*}
endowed with the Sobolev norm inherited from $H^k(\Omega)$, i.e.
\begin{equation*}
\|w\|_{\mathcal{X}^k_j(\Omega)} := \|w\|_{H^k(\Omega)}, \qquad w \in \mathcal{X}^k_j(\Omega).
\end{equation*}
In particular, we use the following cases frequently in the paper:
\begin{align*}
H_0^1(\Omega) = \mathcal{X}^1_0(\Omega), \ \ \
H_0^2(\Omega) = \mathcal{X}^2_1(\Omega), \  \ \
\mathcal{V}^2(\Omega) = \mathcal{X}^2_0(\Omega), \ \ \
\mathcal{V}^3(\Omega) = \mathcal{X}^3_1(\Omega).
\end{align*}

Further details regarding these spaces and their properties can be found in \cite{lcevanspartial2010}. Classical elliptic regularity results under homogeneous Dirichlet boundary conditions (see \cite{gilbarg1998elptic}, Corollary 8.7 and Theorem 8.12) yield the estimate:
	\begin{eqnarray}\label{8}
		\| w\|_{H^{2}(\Omega)}&\leq& \sqrt{C_1} \|\Delta w\|_{L^{2}(\Omega)},
	\end{eqnarray} 
and by the Poinca\'re inequality, we also have 
	\begin{eqnarray}\label{8.1}
		\| w\|_{H_0^1(\Omega)}\leq \sqrt{C_2} \|\nabla w\|_{L^{2}(\Omega)},
	\end{eqnarray} 
	where $ C_1$ and $ C_2$ are positive constants depending only on the domain.  
Moreover, the following general result is known:   
\begin{lemma}[see, \cite{KW}]\label{EN}
	Let $\Omega$ be a bounded smooth domain in $\mathbb{R}^n$ and $k \in \mathbb{N}$. There exists a constant $C_{k,n}>0$ such that for all $w \in H^{k+2}(\Omega)$ and $\frac{\partial w}{\partial\nu}\big|_{\partial \Omega}=0,$ it holds that
	\begin{eqnarray}\label{ES1}
		\|w\|_{H^{k+2}(\Omega)} \leq C_{k,n} \left(\|w\|_{L^2(\Omega)}+ \|\Delta w\|_{H^k(\Omega)}\right).	
	\end{eqnarray}
\end{lemma}
In the case of the boundary condition $w\big|_{\partial \Omega}=0,$ doing integration by parts and invoking the Poinca\'re inequality $\| w\|_{L^2(\Omega)}\leq C_P \|\nabla w\|_{L^{2}(\Omega)}$, we have:
	\begin{eqnarray*}
		\|\nabla w\|_{L^{2}(\Omega)} \leq C_P \|\Delta w\|_{L^{2}(\Omega)}\leq C_P \|\Delta w\|_{H^{k}(\Omega)}.
	\end{eqnarray*} 
Thus, from \eqref{ES1}, there exists a constant $C_1>0$ such that for all $w \in H^{k+2}(\Omega),$ we get 
\begin{eqnarray}\label{88}
		\|w\|_{H^{k+2}(\Omega)} \leq \sqrt {C_1} \|\Delta w\|_{H^k(\Omega)}, \ \ \  \forall k\in \mathbb N.	
	\end{eqnarray}
The embedding $H^2(\Omega)\hookrightarrow L^{\infty}(\Omega)$ is continuous, and there exists a constant $\widetilde{C}>0$ such that
	\begin{equation}\label{S-EI}
	\|w\|_{L^\infty(\Omega)}^2 \leq \widetilde{C}\, \|w\|_{H^2(\Omega)}^2.	
	\end{equation}
	 
	The inner product in  $ L^2(\Omega)$ is denoted by $ (\cdot,\cdot)$. For the duality pairing between $H_0^m(\Omega)$ and it's dual $H^{-m}(\Omega),$ we use $ \langle\cdot,\cdot\rangle$, and $\|\cdot\|_{H^{-m}(\Omega)}$ denotes the corresponding dual norm.
	
	%\noindent We will frequently use the following inequality throughout the paper. 
	%\begin{lemma}\label{9..}(Gronwall's inequality) (see, \cite{salsa2015}, section 10.3.2 ) Let $\xi,$ $\mathcal{G}$  be continuous functions on $ [0,T],$ with $\mathcal{G}$ non-decreasing and let $ C>0$ a positive constant. If $ \xi(t) \leq \mathcal{G}(t)+ C \int_0^t \xi(s) ds, \  \mbox{for all} \ t \in [0,T],$ then
	%	\begin{eqnarray*}
%			\xi(t) \leq \mathcal{G}(t)\exp(C\,t), \  \mbox{for all} \ t \in [0,T].
	%	\end{eqnarray*} Moreover, when $\xi$ is non-negative and $\xi(t) \leq C %\int_0^t \xi(s)\, ds, \ \forall \, t \in [0,T],$ then $\xi(t)=0$ on $[0,T]$. 
	%\end{lemma}

%	\begin{lemma}\label{ag} (Agmon's inequality, \cite{agmon2}, Lemma 13.2)
%		Let $\Omega \subset \mathbb{R}^n$, and $ u\in H^m(\Omega)$ with $ m>\frac{n}{2}.$ Then there exists a constant $ \gamma>0$ which depending on $\Omega$ and $m$ such that 
%		\begin{eqnarray}
%			\Vert u \Vert _{L^{\infty}(\Omega)}\leq \gamma \Vert u \Vert_{H^m(\Omega)}^{\frac{n}{2m}} \ \Vert u \Vert_{L^2(\Omega)}^{1-(\frac{n}{2m})}.
%		\end{eqnarray} 
%	\end{lemma} 
%	\noindent In particular, if $n=2$ and $ m=2,$ we obtain
%	\begin{eqnarray}\label{a123}
%		\Vert u \Vert _{L^{\infty}(\Omega)}\leq \gamma \Vert u \Vert_{H^2(\Omega)}^{\frac{1}{2}} \ \Vert u \Vert_{L^2(\Omega)}^{\frac{1}{2}}. 
%	\end{eqnarray}

	\subsection{Existence and uniqueness of weak solution}\label{S31}
	In this subsection, we introduce the assumptions on the coefficients and source terms required to establish the existence and uniqueness of the weak solution to the direct problem using the Faedo-Galerkin method.	
	\begin{Ass}\label{a2}
		We assume that
		\begin{itemize}
			\item the constant $r$ is strictly positive;
			\item the coupling coefficient $\alpha(x)$ is bounded above by a positive constant $\alpha_1$ for all $x \in \Omega$;
			\item the source terms satisfy $F \in L^{2}(0,T;L^{2}(\Omega))$ and $G \in L^{2}(0,T;L^{2}(\Omega))$.
		\end{itemize}
	%	\begin{eqnarray*} 
	%		\left \{ \begin{array}{ll}
	%			r \ \mbox{is a positive constant,}\\
	%			\alpha \leq \alpha_1, \ \forall x \in \Omega, \ \mbox{and} \ \alpha_1>0 \ \mbox{is \ a \ constant}   \\
	%			F\in L^2(0,T;L^2(\Omega)), \ G\in L^2(0,T;L^2(\Omega)).
	%		\end{array} \right.
	%	\end{eqnarray*}
	\end{Ass}
	\begin{definition}\label{4}
		A function pair $ (u,\theta) \in L^2(0,T;H_0^2(\Omega))\times L^2(0,T;H_0^1(\Omega))$ with  $ u_t\in L^2(0,T;H_0^1(\Omega)), $ $u_{tt}\in L^2(0,T;H^{-2}(\Omega)), (u-r\Delta u)_{tt}\in L^2(0,T;H^{-2}(\Omega))$ and $\theta_{t}\in L^2(0,T;H^{-1}(\Omega))$  is called a weak solution of \eqref{1a}-\eqref{1d} if 
		\begin{enumerate}
			\item for each  $v\in H_0^2(\Omega), w\in H_0^{1}(\Omega),  \mbox{and almost every} \ t\in[0,T]$, the following equalities holds:
			\begin{subequations}
				\begin{alignat}{4}
					&\langle \left(u(t)-r\Delta u(t)\right)_{tt}, v \rangle +(\Delta u(t), \Delta v)-(\alpha\nabla \theta (t), \nabla v)=(F(t),v) \label{8.111a}\\  
					&	\langle \theta_{t}(t),w \rangle +(\nabla \theta(t),\nabla w)+ (\alpha\nabla u_{t}(t),\nabla w)=(G(t),w), \label{8.111b} 
				\end{alignat}
			\end{subequations}
			\item {$  \ u(0)=u_0, \ u_{t}(0)=v_0, \ \mbox{and} \ \theta(0)=\theta_0.$}
		\end{enumerate}
	\end{definition}
	
	\begin{remark}\label{r1}
		From the regularity required in Definition \ref{4}, we infer that $ (u,\theta)\in H^1(0,T;H_0^1(\Omega))\times H^1(0,T;H^{-1}(\Omega))$. Consequently, $ (u,\theta)$ belongs to $ C([0,T];H_0^1(\Omega) )\times C([0,T];H^{-1}(\Omega))$ and $ u_t \in C([0,T];H^{-2}(\Omega)).$ Thus the equality $ u(0)=u_0, u_{t}(0)=v_0$ and $\theta(0)=\theta_0$ can be justified in the trace sense. 
	\end{remark}
	
	\noindent \textbf{Faedo–Galerkin approximation.}
	We now apply the  Faedo-Galerkin  method to prove the solvability of \eqref{1a}-\eqref{1d}. Let $\{ \xi_{m}\}_{m=1}^{\infty}$ be an orthonormal basis of $L^{2}(\Omega)$ consisting of the eigenfunctions corresponding to the eigenvalues $\lambda_k$ of the biharmonic operator $\Delta^2$ subject to clamped boundary condition $\xi_{m}= 0$ on $\Gamma$ and $ \frac{\partial \xi_m}{\partial \textnormal{n}}=0$ on $\Gamma$. These eigenfunctions are smooth and form an orthogonal basis for $H_0^2(\Omega)$.
	
	Moreover, 	$\{\chi_m\}_{m=1}^{\infty}$ denotes the normalized eigenfunctions of $-\Delta$ corresponding to the eigenvalue $\beta_{k}$  in $H_0^1(\Omega)$, which form an orthogonal basis of $H_0^1(\Omega)$ and an orthonormal basis for $L^2(\Omega)$. Moreover, consider $\mathbb{P}_n: L^2(\Omega)\to W_n$ and $\mathbb{Q}_n:L^2(\Omega)\to V_n$ to be the orthogonal projection on the space $W_n=span\{\xi_1,\xi_2,...,\xi_n\}$ and $V_n=span\{\chi_1,\chi_2,...,\chi_n\}$, respectively. 
	
	Let $u_n(x,t)$ and $\theta_n(x,t)$ be the Galerkin approximation of the form 
	$	u_{n}(t)=\sum_{m=1}^{n} r_{m,{n}}(t)\xi_{m}, \ \theta_n(t)= \sum_{m=1}^{n} d_{m,n}(t) \chi_m. $
	Then, consider the following Galerkin form equivalent to a set of ordinary differential equations (ODEs):
	\begin{subnumcases}{\label{2}}
		\left( (u_{n}(t)-r\Delta u_{n}(t))'' , \xi_m \right)
		+ (\Delta u_{n}(t), \Delta \xi_m)
		- (\alpha \nabla \theta_{n}(t), \nabla \xi_m)
		= (F(t), \xi_m),  \label{2a} \\[4pt]
		(\theta_{n}'(t), \chi_m)
		+ (\nabla \theta_{n}(t), \nabla \chi_m)
		+ (\alpha \nabla u_{n}'(t), \nabla \chi_m)
		= (G(t), \chi_m), \label{2b} \\[4pt]
		u_{n}(0) = u_{0,n}, \quad
		u_{n}'(0) = v_{0,n}, \quad
		\theta_{n}(0) = \theta_{0,n}. \label{2c}
	\end{subnumcases}
	where $u_{0,n}=\mathbb{P}_n(u_0)$, $v_{0,n}=\mathbb{P}_n(v_0)$ and $\theta_{0,n}=\mathbb{Q}_n(\theta_0)$.
	
	Next, by considering the source functions $F,G\in C([0,T];L^2(\Omega))$, and the existence theory of ODEs (see, \cite{PH}), we can obtain that there exists a local-in-time regular solution to the system \eqref{2a}-\eqref{2c}. Furthermore, by deriving an appropriate \textit{a priori estimate}, we can prove that the solution does not blow up before the final time $T$.

%	with corresponding expansions fo the initial data
%	$$u_{0,n}=\sum_{m=1}^{n} e_{m,{n}}\xi_{m}, \ \ v_{0,n}=\sum_{m=1}^{n} f_{m,{n}}\xi_{m}, \ \ \  \text{and}\ \ \ 	\theta_{0,n}= \sum_{m=1}^{n} g_{m,n} \chi_m,$$		
%	\begin{eqnarray*} 
%		u_{n}(t)&=&\sum_{m=1}^{n} r_{m,{n}}(t)\xi_{m}, \ \ \theta_n(t)= \sum_{m=1}^{n} d_{m,n}(t) \chi_m, \\ 
%		u_{0,n}&=&\sum_{m=1}^{n} e_{m,{n}}\xi_{m}, \ \ v_{0,n}=\sum_{m=1}^{n} f_{m,{n}}\xi_{m} \ \
%		\theta_{0,n}= \sum_{m=1}^{n} g_{m,n} \chi_m
%	\end{eqnarray*}
		
	Therefore, there exist coefficients $ r_{m,n}(t)$ and $d_{m,n}(t)$  such that $ (u_n(t),\theta_{n}(t)) $ for all  $t\in [0,T]$ that satisfy the system \eqref{2a}-\eqref{2b}. Moreover, by using the density of $C([0,T];L^2(\Omega))$ in $L^2(0,T;L^2(\Omega))$, it is clear that the above result also holds for source functions $F,G\in L^2(0,T;L^2(\Omega))$.

	\begin{theorem}\label{t1}
		Suppose the initial data $ u_{0}\in H_0^2(\Omega)$, $v_0\in H_0^1(\Omega)$ and $ \theta_0 \in L^2(\Omega)$. Then under Assumption \ref{a2}, the direct problem \eqref{1a}-\eqref{1d} has a unique weak solution based on Definition \ref{4}. Furthermore, the following energy estimates hold:
		\begin{eqnarray}
			\Vert  (u,\theta)\Vert^2_{L^{2}(0,T;H_0^2(\Omega))\times L^2(0,T;H_0^1(\Omega))}&\leq& \left(C_1T+\frac{C_2}{2}\right) \exp(T) R_1(\theta_0,u_0,v_0,F,G), \label{5.4}\\ 
			\!\!\!\!\!\!\!\Vert (u_{t},\theta_{t})\Vert^2_{L^\infty(0,T;H_0^1(\Omega))\times L^2(0,T;H^{-1}(\Omega))}\!\!\!&\leq& \!\!\!3\left(\frac{5}{6}\exp(T)+\frac{\exp(T)}{3r}\!+\!\alpha_1^2 \frac{T\exp(T)}{r} \!+\!1\right)R_1(\theta_0,u_0,v_0,F,G),
			\label{ke}
			%			\Vert \theta \Vert^2 _{L^\infty(0,T;L^2(\Omega))}	&\leq& \left(C_e+1\right)R_1(\theta_0,u_0,v_0,F,G),
		\end{eqnarray}
%		and
%		\begin{eqnarray}
%			\Vert \nabla u_{t}\Vert^2_{L^2(0,T; L^2(\Omega))}&\leq &  \frac{T}{r} (C_e+1) \ R_1(\theta_0,u_0,v_0,F,G), \label{5.4.1} \\ 
%			\Vert \left(u-r\Delta u\right)_{tt}\Vert^2_{L^2(0,T;H^{-2}(\Omega))}
%			&\leq& C_4 R_1(\theta_0,u_0,v_0,F,G), \label{7.8.9}
%		\end{eqnarray}
and 
\begin{eqnarray}
			\Vert\left( u-r\Delta u\right)_{tt}\Vert^2_{L^2(0,T;H^{-2}(\Omega))}\leq  3\left( T \exp(T)+\alpha_1^2 \frac{\exp(T)}{2}+1\right) R_1 (\theta_0, u_0, v_0, F,G), \label{utt}
\end{eqnarray}
where the constants $ C_1$ and $C_2$ are coming from estimates (\ref{8}) and (\ref{8.1}) respectively, and  
		\begin{align}\label{r1q}
			R_1(\theta_0,u_0,v_0,F,G)= &\Vert F \Vert^2_{L^2(0,T;L^2(\Omega))} + \Vert G\Vert^2_{L^2(0,T;L^2(\Omega))}\nonumber\\
		  & + \|v_0\|_{L^2(\Omega)}^2 +r\Vert \nabla v_{0}\Vert^2_{L^2(\Omega)}+\Vert \Delta u_{0}\Vert^2_{L^2(\Omega)}+\Vert \theta_{0}\Vert^2_{L^2(\Omega)}.
		\end{align}
	\end{theorem}
	\begin{proof}
		%\noindent {\bf A Priori estimates:}
		First, we multiply the equations \eqref{2a} and \eqref{2b} by $r_{m,n}'(t)$ and $ d_{m,n}(t)$, respectively, and take the sum over $m = 1,2,...,n$. Then, by integrating over $(0,t)$, applying integration by parts, and using the initial conditions \eqref{2c}, we get the following:		
		\begin{eqnarray}
			&&\frac{1}{2}\int_{\Omega} |u_{n}'(t)|^2 dx+\frac{r}{2} \int_{\Omega} |\nabla u_{n}'(t)|^2 dx - \int_{0}^{t} \int_{\Omega}\alpha(x)\nabla \theta_{n}(\tau)\cdot\nabla u_{n}'(\tau) \,dx \,d\tau +\frac{1}{2}\int_{\Omega}|\Delta u_{n}(t)|^2 dx \label{6}\\ 
			&&\hspace{2cm} = \int_0^t \int_{\Omega} F(x,\tau) u_n'(\tau) \,dx\, d\tau+\frac{1}{2}\int_{\Omega} |v_{0,n}|^2 dx +\frac{r}{2}\int_{\Omega}|\nabla v_{0,n}|^2 dx + \frac{1}{2}\int_{\Omega} |\Delta u_{0,n}|^2 dx, \nonumber \\
			&&\frac{1}{2}\int_{\Omega} |\theta_{n}(t)|^2 dx +\int_{0}^t \int_{\Omega}|\nabla \theta_{n}(\tau)|^2 \,dx\, d\tau  +\int_{0}^{t} \int_{\Omega}\alpha(x) \nabla u_{n}'(\tau)\cdot\nabla\theta_{n}(\tau) \,dx\, d\tau \nonumber \\
			&&\hspace{2cm}=  \int_0^t \int_{\Omega} G(x,\tau) \theta_n (\tau) \,dx\, d\tau+\frac{1}{2}\int_{\Omega} |\theta_{0,n}|^2 dx, \ \ \ \ \ \text{for all } t\in [0,T].\label{5}
		\end{eqnarray}
		Next, adding equations \eqref{6} and \eqref{5}, and applying Cauchy's inequality for the right-hand side source terms, we obtain
		\begin{eqnarray*}
			\lefteqn{\int_{\Omega} |u_{n}'(t)|^{2} dx +\int_{\Omega}|\theta_{n}(t)|^2 dx+r\int_{\Omega}|\nabla u_{n}'(t)|^2 dx +\int_{\Omega}|\Delta u_{n}(t)|^2 dx  +2\int_{0}^t\int_{\Omega}\vert\nabla\theta_{n}(\tau)\vert^2 dx d\tau} \nonumber\\
			&\leq& \int_0^t \int_{\Omega} |F(x,\tau)|^2 dx d\tau+ \int_{0}^t\int_{\Omega} |G(x,\tau)|^2 dx d\tau  +\int_0^t \int_{\Omega}\left( |u_n'(\tau)|^2 +|\theta_n(\tau)|^2\right)dx d\tau\nonumber \\ 
			&&\ \ \ \ +\int_{\Omega} |v_{0,n}|^2 dx+r\int_{\Omega}|\nabla v_{0,n}|^2 dx + \int_{\Omega} |\Delta u_{0,n}|^2 dx+\int_{\Omega} |\theta_{0,n}|^2 dx  \ \ \ \ \text{for all } t\in [0,T] . 
		\end{eqnarray*}
		Now, applying Gronwall's inequality, we get the following estimate for each $t\in [0,T]$:
		\begin{align*}
			\int_{\Omega} |u_{n}'(t)|^{2} dx+\int_{\Omega}|\theta_{n}(t)|^2 dx +r\int_{\Omega}|\nabla u_{n}'(t)|^2 dx & +\int_{\Omega}|\Delta u_{n}(t)|^2 dx  +2\int_{0}^t\int_{\Omega}\vert\nabla\theta_{n}(\tau)\vert^2 dx d\tau\nonumber\\
			& \leq \exp(T)\  R_1 (\theta_0, u_0, v_0, F,G). 
		\end{align*}
		We can derive the following norm estimates on $u_n$, $u_n'$ and $\theta_n$: 
		\begin{eqnarray}
			\max_{t\in[0,T]}\Vert u_{n}'(t)\Vert^2_{L^2(\Omega)}&\leq& \exp(T)  R_1 (\theta_0, u_0, v_0, F,G),\nonumber\\[1pt]
			\Vert u_{n}'\Vert^2_{L^2(0,T;L^2(\Omega))}&\leq&   T \exp(T)  R_1 (\theta_0, u_0, v_0, F,G) \nonumber,\\
			\Vert \nabla\theta_{n}\Vert^2_{L^2(0,T;L^2(\Omega))}&\leq& \frac{\exp(T)}{2} R_1 (\theta_0, u_0, v_0, F,G) , \label{7.11}
			\\ [1pt]
			\max_{t\in[0,T]}\Vert \nabla u_n'(t) \Vert_{L^2(\Omega)}^2&\leq& \frac{\exp(T)}{r}  R_1 (\theta_0, u_0, v_0, F,G), \label{mvc1}
			\\[1pt]
			\max_{t\in[0,T]} \Vert \Delta u_{n}(t)\Vert^2_{L^2(\Omega)}&\leq&  \exp(T) R_1 (\theta_0, u_0, v_0, F,G) ,\label{7.13} \\ [1pt]
			\Vert \Delta u_{n}\Vert^2_{L^2(0,T;L^2(\Omega))}&\leq&  T \exp(T) R_1 (\theta_0, u_0, v_0, F,G) .\label{27}		
		\end{eqnarray}
		By proceeding in the same manner as in Theorem 2.4 of \cite{anjuna}, and invoking Assumption \ref{a2}, we also obtain the estimate
		\begin{equation}\label{du1}
			\Vert \big(u_n(t)-r \Delta u_n(t)\big)^{\prime \prime}\Vert_{H^{-2}(\Omega)} \leq \Vert \Delta u_n(t)\Vert_{L^2(\Omega)}+\alpha_1\Vert \nabla \theta_n(t)\Vert _{L^2(\Omega)} +\Vert F(t)\Vert_{L^2(\Omega)},
		\end{equation}
		and 
		\begin{eqnarray}\label{du2}
			\Vert \theta_n'(t)\Vert_{H^{-1}(\Omega)} \leq  \Vert \nabla \theta_n(t)\Vert_{L^2(\Omega)}+\alpha_1 \Vert \nabla u_n'(t) \Vert_{L^2(\Omega)}+\Vert G(t)\Vert_{L^2(\Omega)}.
		\end{eqnarray}
		%where $ C_3= \max\{1,\alpha_1\}.$ 
		Now squaring on both sides in equations \eqref{du1} and \eqref{du2}, integrating over $ [0,T],$ and using the estimates \eqref{7.11}-\eqref{27}, we get
		\begin{eqnarray*}
			\Vert\left( u_n-r\Delta u_n\right)^{\prime \prime}\Vert^2_{L^2(0,T;H^{-2}(\Omega))}&\leq&  3\left( T \exp(T)+\alpha_1^2 \frac{\exp(T)}{2}+1\right) R_1 (\theta_0, u_0, v_0, F,G),  \\ 
			\Vert \theta_n'\Vert^2_{L^2(0,T;H^{-1}(\Omega))}&\leq&
			 3\left( \frac{\exp(T)}{2}+\alpha_1^2 \frac{T\exp(T)}{r}+1\right) R_1 (\theta_0, u_0, v_0, F,G).
		\end{eqnarray*} 
%		where $ C_4=3 C_3^2 \left( \bigg(C_e+1\bigg)\bigg(T+\frac{1}{2}\bigg)+1\right),$  		$C_5= 3 C_3^2 \left( \bigg( C_e+1\bigg) \bigg(\frac{T}{r}+\frac{1}{2}\bigg)+1\right).$ \\

	Next, if we consider the eigenvalues $\beta_k$ corresponding to the eigenfunctions $\chi_k$ of the $-\Delta$ operator, then each of the eigenvalues $1+r\beta_k$ of the operator $I-r\Delta$ satisfies $1+r\beta_k \geq 1$. Therefore, the operator $I-r\Delta$ is invertible. Moreover, its norm satisfies the inequality $\|(I-r\Delta)^{-1}\|\leq 1$. Using this norm estimate, we find the following bound for $u_n''$:
\begin{eqnarray*}
    \|u_n''\|_{L^2(0,T;H^{-2}(\Omega))} = \|(I-r\Delta)^{-1} (I-r\Delta)u_n''\|_{L^2(0,T;H^{-2}(\Omega))}\leq \|(u_n-r\Delta u_n)''\|_{L^2(0,T;H^{-2}(\Omega))}.
\end{eqnarray*}  	
%		$\left( u_n-r\Delta u_n\right)^{\prime \prime} \in L^2(0,T;H^{-2}(\Omega))\leq \|(u_n-r\Delta u%_n)''\|_{L^2(0,T;H^{-2})},$ 
%		we have \begin{eqnarray}
%			\Vert u_{n}^{\prime \prime}-r\Delta u_n^{\prime \prime}\Vert_{L^2(0,T;H^{-2}(\Omega))} &=& \Vert (1-r\lambda_n) u_{n}^{\prime \prime}\Vert_{L^2(0,T;H^{-2}(\Omega))}\ \nonumber \\ \leq& C_4 R_1 (\theta_0, u_0, v_0, F,G),
%		\end{eqnarray}
%		then $ u_n^{\prime \prime}\in L^2(0,T;H^{-2}(\Omega)).$

		Thus, from the estimates obtained above, the equality of norms estimate \eqref{8} and the Banach-Alagolu weak compactness theorem, we deduce that there exist a subsequence $\{u_{nk}\}$ of $ \{ u_n\}$, $ \{\theta_{nk}\}$ of $ \{\theta_n\}$ and functions $ u\in L^2(0,T;H_0^2(\Omega)),$ $\theta \in L^2(0,T;H_0^1(\Omega)),$ such that 
		\begin{eqnarray*}
			\left \{ \begin{array}{lclccl}
				u_{n_{k}}&\rightharpoonup& u \ &\mbox{weakly  in}  \ \ L^{2}(0,T;H_0^{2}(\Omega)), \\[.5pt]
				\Delta u_{n_k}&\rightharpoonup&\Delta u\ &\mbox{weakly  in} \ \ L^{2}(0,T;L^{2}(\Omega)),\\[.5pt]
				\nabla u_{n_k}'&\rightharpoonup& \nabla u'\ &\mbox{weakly  in} \ \ L^{2}(0,T;L^{2}(\Omega)), \\[.5pt]
				\left(u_{n_{k}}-r\Delta u_{n_{k}}\right)^{\prime \prime} &\rightharpoonup& \left(u-r\Delta u\right)^{\prime \prime} &\mbox{\ weakly  in} \ L^{2}(0,T;H^{-2}(\Omega)), \\ [.5pt] 
				u_{n_{k}}^{\prime \prime} &\rightharpoonup& u^{\prime \prime} \ \ &\mbox{\  weakly   in} \  L^{2}(0,T;H^{-2}(\Omega)),\\[.5pt]
				\theta_{n_{k}}&\rightharpoonup& \theta \ &\mbox{\ weakly  in}  \ \ L^{2}(0,T;H_0^1(\Omega)), \\[.5pt]
				\nabla \theta_{n_k}&\rightharpoonup&\nabla \theta\ &\mbox{weakly  in} \ \ L^{2}(0,T;L^{2}(\Omega)),\\[.5pt]
				\theta_{n_{k}}^{ \prime} &\rightharpoonup & \theta^{ \prime} \ \ \ & \mbox{\ \ weakly in} \ \ L^{2}(0,T;H^{-1}(\Omega)),
			\end{array} \right.
		\end{eqnarray*}
		as $k \to \infty.$ Applying the standard limit arguments to the equation \eqref{2a}-\eqref{2b}, we derive a weak solution to the direct problem \eqref{1a}-\eqref{1d} which satisfies the estimates \eqref{5.4} and \eqref{ke}. The uniqueness of the solution to the direct problem and verification of initial data can be done by following the steps provided in the proof of Theorem 1 in \cite{anjuna:2021}. Hence the proof.
	\end{proof}
	%\subsection{Regularizing effect of Kelvin-Voigt damping}
	%In this section, we examine the impact of Kelvin-Voigt damping on the regularity of the solution to the direct problem by incorporating Kelvin-Voigt damping into the physical model (see, []). Motivated by previous studies that explored the effects of various types of damping in beam-plate models (see, []), we provide an analysis of the regularizing effect of Kelvin-Voigt damping in this subsection.  Consider the following anisotropic thermoelastic plate equation with Kelvin-Voigt damping (see, [])
	%  
	%	\begin{subnumcases}{}\label{1.1}
		%	&		$ u_{tt}-r\Delta u_{tt}+ \kappa \Delta^2u_t +\Delta ^2 u+ \mbox{div}(\alpha(x).\nabla \theta) = F(x,t), \;\; \; (x,t)\in \Omega_T,$   \label{1.1a} \\ [1pt]
		%	&$\theta_{t}- \Delta\theta -\mbox{div}( \alpha(x). \nabla u_t)= G(x,t),\qquad \qquad \quad \ (x,t)\in \Omega_T,$
		%	\label{1.1b} \\ [1pt]
		%	&$ u=\frac{\partial u}{\partial\nu}=0, \ \theta =0 \qquad \qquad\; \qquad \qquad \qquad \qquad \quad (x,t)\in \Gamma_T,$ \\ [1pt]
		%	&	 $	u(x,0) = u_0, u_{t}(x,0)  = v_0, \; \theta(x,0)= \theta_0 \qquad \qquad \qquad x  \in \Omega,$ \nonumber,
		%\end{subnumcases}
		%where $ \kappa \Delta^2u_t$ is the Kelvin-Voigt damping term.\\
		%\textbf{Regularity of weak solution;}
		%
		%\begin{theorem}
		%	Let assumption \ref{a2} hold true. Then there exist a unique solution \\ $ u\in L^2(0,T; \mathcal{V}^2(\Omega)), \ u_t \in L^2(0,T;\mathcal{V}^2(\Omega)), \ u_{tt}\in L^2(0,T;H_0^1(\Omega)), \\ \theta \in L^2(0,T;H_0^1(\Omega)), \ \theta_t \in L^2(0,T;H^{-1}(\Omega)).$  
		%\end{theorem}

	\subsection{Regularity of weak solutions}\label{S-RWS}
		This section establishes the regularity of weak solutions to the direct problem by imposing stronger regularity constraints on the thermal expansion coefficient $\alpha(x)$ and the initial data. 
		\begin{theorem}\label{T-RWS}
			Suppose the initial data $ u_0 \in \mathcal{V}^{3}(\Omega),$ $v_0 \in H_0^2(\Omega)$ and $ \theta_0\in H_0^1(\Omega)$. Assume that Assumption \ref{a2} holds and, in addition, the thermal expansion coefficient $ \alpha \in H^{3}(\Omega)$. Then the corresponding weak solution to the system \eqref{1a}-\eqref{1d} attains the following higher regularity	
			\begin{align*}
				&(u,\theta)\in L^\infty(0,T;\mathcal{V}^3(\Omega))\times L^2(0,T; \mathcal{V}^2(\Omega)),\\
				(u_t,\theta_t)\in  &L^\infty(0,T;H_0^2(\Omega))\times L^2(0,T;L^2(\Omega))  \ \ \text{and}\ \   u_{tt}\in L^2(0,T;H_0^1(\Omega)).
			\end{align*}									
			\end{theorem}
		
		\begin{proof}
			In order to derive higher-order regularity estimates for the weak solutions $(u, \theta)$ of system \eqref{1aa}, we again work with the Galerkin approximated system \eqref{2}. 
			%But, in order to ease the estimations, we will derive this regularity estimate in a formal manner by directly working on system \eqref{1aa}. 
			
            We first multiply the equations \eqref{2a} and \eqref{2b} by $-\sqrt {\lambda_m} r_{m,n}'(t)$ and $\beta_m d_{m,n}(t)$, respectively and take the sum over $m = 1,2,...,n$.. Then integrating both resulting identities over $(0,t),$  
        applying integration by parts and utilizing the  homogeneous boundary conditions \eqref{1c} and initial condition \eqref{1d}, we obtain 			
			\begin{align}\label{cx}
				\frac{1}{2} \int_{\Omega} |\nabla u_n'(t)|^2& dx + \frac{r}{2} \int_{\Omega} |\Delta u_n'(t)|^2 dx + \frac{1}{2} \int_{\Omega} |\nabla\Delta u_n(t)|^2 dx = - \int_0^t\int_{\Omega}  F(x,\tau) \Delta u_n'(\tau) dx d\tau  \nonumber \\
				 &+ \int_0^t\int_{\Omega}  \alpha(x)\Delta \theta_n(\tau) \Delta u_n'(\tau) dx d\tau +\int_0^t\int_{\Omega} \big(\nabla \alpha(x)\cdot\nabla \theta_n(\tau)\big)\Delta u_n'(\tau) dx d\tau \nonumber \\ 
				 &+ \frac{1}{2}\int_{\Omega} |\nabla v_{0,n}|^2 dx +\frac{r}{2}\int_{\Omega} |\Delta v_{0,n}|^2 dx + \frac{1}{2}\int_{\Omega} | \nabla\Delta u_{0,n}|^2dx,
			\end{align}
			\begin{align}\label{cx1}
				 \frac{1}{2} \int_{\Omega} |\nabla \theta_n(t) |^2 dx &+\int_0^t\int_{\Omega} |\Delta \theta_n(\tau)|^2 dx d\tau  =\frac{1}{2}\int_{\Omega}|\nabla \theta_{0,n}|^2 dx -\int_0^t\int_{\Omega} G(\tau)\Delta \theta_n(\tau) dx d\tau   \nonumber \\ 
				& - \int_0^t\int_{\Omega}  \alpha(x) \Delta u_n'(\tau) \Delta \theta_n(\tau) dx d\tau - \int_0^t\int_{\Omega} \big( \nabla \alpha(x)\cdot \nabla u_n'(\tau)\big) \Delta \theta_n(\tau)dx d\tau,
			\end{align} 
			where we used the following conditions on the boundary $\Gamma:$ $$\frac{\partial u_n'}{\partial\nu}=\sum_{m=1}^{n}r'_{m,n}(t)\frac{\partial \xi_m}{\partial\nu}=0, \ \ \frac{\partial \Delta u_n}{\partial\nu}=-\sum_{m=1}^{n} r_{m,n}(t) \lambda_m^{\frac{1}{2}} \frac{\partial \xi_m}{\partial\nu}=0, \ \  \Delta\theta_n=-\sum_{m=1}^{n} d_{m,n}(t)\beta_m \chi_m=0.$$

			Next, we add estimates \eqref{cx} and \eqref{cx1}, apply Young’s inequality, and implement the embedding  $H^2(\Omega)\hookrightarrow L^{\infty}(\Omega)$ along with the estimate $ \Vert \nabla \alpha\Vert^2_{L^{\infty}(\Omega)}\leq \widetilde{C}\,\Vert \nabla \alpha \Vert^2_{H^2(\Omega)}\leq \widetilde{C}\, \Vert \alpha \Vert^2_{H^3(\Omega)}$ to get
			\begin{align*}
				\int_{\Omega} |\nabla &u_n' (t)|^2 dx + r \int_{\Omega} |\Delta u_n'(t)|^2 dx + \int_{\Omega} |\nabla\Delta u_n(t)|^2 dx+ \int_{\Omega} |\nabla \theta_n(t) |^2 dx +2\int_0^t\int_{\Omega} |\Delta \theta_n(\tau)|^2 dx d\tau \nonumber \\
				  \leq& \ 2 \epsilon\int_0^t \int_{\Omega} |\Delta \theta_n(\tau)|^2 dx d\tau + 2 \int_0^t \int_{\Omega} |\Delta u_n'(\tau)|^2 dx d\tau  + \widetilde{C}\, \frac{1}{2}\Vert \alpha \Vert^2_{H^3(\Omega)}\int_0^t\int_{\Omega} |\nabla \theta_n(\tau)|^2 dx d\tau  \nonumber \\
				&+\frac{1}{4\epsilon}\left( \int_0^t\int_{\Omega} |G(\tau)|^2 dx d\tau+ \widetilde{C}\, \Vert \alpha\Vert^2_{H^3(\Omega)}\int_0^t \int_{\Omega}|\nabla u_n'(\tau)|^2 dx d\tau\right)+\frac{1}{2}\int_0^t \int_{\Omega} |F(\tau)|^2 dx d\tau \nonumber \\
				&+\int_{\Omega} |\nabla v_{0,n}|^2 dx   +r\int_{\Omega} |\Delta v_{0,n}|^2 dx+\int_{\Omega} | \nabla\Delta u_{0,n}|^2dx+\int_{\Omega}|\nabla \theta_{0,n}|^2 dx.
			\end{align*}
			Now, we choose $ \epsilon=\frac{1}{2}$ and apply the projection inequalities $\|\nabla v_{0,n}\|^2_{L^2(\Omega)}\leq \|\nabla v_0\|^2_{L^2(\Omega)}$, $\|\Delta  v_{0,n}\|^2_{L^2(\Omega)}\leq \|\Delta v_0\|^2_{L^2(\Omega)}$, $\|\nabla\Delta u_{0,n}\|^2_{L^2(\Omega)}\leq \|\nabla \Delta u_0\|^2_{L^2(\Omega)}$ and $\|\nabla \theta_{0,n}\|^2_{L^2(\Omega)}\leq \|\nabla \theta_0\|^2_{L^2(\Omega)}$ to deduce 
				\begin{align*}
				\int_{\Omega} |\nabla u_n' (t)|^2 dx +& r \int_{\Omega} |\Delta u_n'(t)|^2 dx + \int_{\Omega} |\nabla\Delta u_n(t)|^2 dx+ \int_{\Omega} |\nabla \theta_n(t) |^2 dx +\int_0^t\int_{\Omega} |\Delta \theta_n(\tau)|^2 dx d\tau \nonumber \\
				\leq \  \widetilde{C}\, \frac{1}{2}&\Vert \alpha\Vert^2_{H^3(\Omega)} \left( \int_0^t \int_{\Omega}|\nabla u_n'(\tau)|^2 dx d\tau   +\int_0^t\int_{\Omega} |\nabla \theta_n(\tau)|^2 dx d\tau\right)  \nonumber \\
				&+  2 \int_0^t \int_{\Omega} |\Delta u_n'(\tau)|^2 dx d\tau+R_2(\theta_0,u_0,v_0,F,G,\alpha), \ \ \ \ \ \text{for all } t\in [0,T],
			\end{align*}
            where 
				\begin{align*}
				&R_2(\theta_0,u_0,v_0,F,G,\alpha) := \frac{1}{2}\int_0^T \int_{\Omega} |F|^2 dx d\tau + \frac{1}{2} \int_0^T\int_{\Omega} |G|^2 dx d\tau \nonumber \\
				&\hspace{1cm}+\int_{\Omega} |\nabla v_0|^2 dx   +r\int_{\Omega} |\Delta v_0|^2 dx+\int_{\Omega} | \nabla\Delta u_0|^2dx+\int_{\Omega}|\nabla \theta_0|^2 dx.
			\end{align*}
			Next, applying Gr\"onwall's inequality, we arrive at the following regularized energy estimate on the solution of the approximated Galerkin system:
			\begin{align}\label{S-RW1}
				\int_{\Omega} |\nabla u_n' (t)|^2 dx +& r \int_{\Omega} |\Delta u_n'(t)|^2 dx + \int_{\Omega} |\nabla\Delta u_n(t)|^2 dx+ \int_{\Omega} |\nabla \theta_n(t) |^2 dx +\int_0^t\int_{\Omega} |\Delta \theta_n(\tau)|^2 dx d\tau \nonumber \\
				&\leq K(\alpha,r,\widetilde{C},T) \ R_2(\theta_0,u_0,v_0,F,G,\alpha),\ \ \ \ \ \text{for all } t\in[0,T].
			\end{align}
	where $K(\alpha,r,\widetilde{C},T):=\exp\left\{\max\left( \frac{\widetilde{C}}{2}\Vert \alpha\Vert^2_{H^3(\Omega)},\frac{2}{r}\right)T\right\}.$ Since, in the above inequality, the right-hand side is independent of $n$, we can extract subsequences of $(u_n,\theta_n)$ and $u'_n$ (again represented as the same)  converging weakly to the functions $(u,\theta)\in L^\infty(0,T;\mathcal{V}^3(\Omega))\times L^2(0,T;\mathcal{V}^2(\Omega)), \ \mbox{and} \ u_{t}\in L^\infty(0,T;H_0^2(\Omega)),$ respectively, through the equality of norms \eqref{8} and \eqref{88}. 
			
			Finally, by using the weak-sequential lower semicontinuity  in equation \eqref{S-RW1}, we get the following norm estimates for $u', u$ and $\theta$:			
			\begin{eqnarray}
				\max_{t\in[0,T]}\Vert \Delta u'(t)\Vert^2_{L^2(\Omega)}&\leq& \frac{1}{r} K(\alpha,r,\widetilde{C},T)  \ R_2(\theta_0,u_0,v_0,F,G,\alpha), \label{cz} \\
				%\Vert \Delta u' \Vert^2_{L^2(0,T;L^2(\Omega))} &\leq& \exp(max\left\{ C \frac{1}{2}\Vert \alpha\Vert^2_{H^3(\Omega)},\frac{2}{r}\right\} \ R_2(\theta_0,u_0,v_0,F,G,\alpha) \\
				\max_{t\in[0,T]}\Vert \nabla\Delta u(t)\Vert^2_{L^2(\Omega))}&\leq& K(\alpha,r,\widetilde{C},T) \ R_2(\theta_0,u_0,v_0,F,G,\alpha), \nonumber  \\
				\Vert \Delta \theta \Vert^2_{L^2(0,T;L^2(\Omega))}&\leq& K(\alpha,r,\widetilde{C},T) \ R_2(\theta_0,u_0,v_0,F,G,\alpha)\nonumber.
			\end{eqnarray}
			
Further, to establish the regularity property of $ u_{tt}\in L^2(0,T;H_0^1(\Omega))$ and $\theta_{t}\in L^2(0,T;L^2(\Omega)),$ we multiply \eqref{2a} by \( r_{m,n}''(t) \) and  \eqref{2b} by \( d_{m,n}'(t) \), summing and integrating over $(0,t)$ to obtain 
			\begin{eqnarray}\label{fs}
				\lefteqn{\int_0^t \int_{\Omega} |u_n''(\tau)|^2 dx d\tau+r \int_0^t\int_{\Omega} |\nabla u_n''(\tau)|^2 dx d\tau= \int_0^t \int_{\Omega} F(x,\tau) u_n'' (\tau) dx d\tau +\int_0^t \int_{\Omega} |\Delta u_n'(\tau)|^2 dx d\tau} \nonumber \\ &&+\int_{\Omega} \Delta u_{0,n}(x)\Delta v_{0,n}(x) dx-\int_{\Omega} \Delta u_n(t)\ \Delta u_n'(t) dx+\int_0^t \int_{\Omega} \alpha(x) \nabla \theta_n(\tau)\cdot\nabla u_n'' (\tau) dx d\tau, \hspace{1in}
			\end{eqnarray}
			and
			\begin{eqnarray}\label{fs1}
				\lefteqn{\int_0^t\int_{\Omega} |\theta_n' (\tau)|^2 dx d\tau +\frac{1}{2}\int_{\Omega} |\nabla \theta_n(t)|^2 dx= \int_0^t\int_{\Omega} G(x,\tau) \theta_n' (\tau) dx d\tau}  \\ &&-\int_0^t\int_{\Omega} \alpha(x) \nabla u_n'' (\tau)\cdot \nabla \theta_n(\tau) dx d\tau+\int_{\Omega}\alpha(x) \nabla u_n' (t) \cdot \nabla \theta_n(t) dx -\int_{\Omega} \alpha(x)\nabla v_{0,n}(x)\cdot \nabla \theta_{0,n}(x) dx, \nonumber
			\end{eqnarray}
			where we also used the condition $ \frac{\partial u_n''}{\partial\nu}=\sum_{m=1}^{n}r''_{m,n}(t)\frac{\partial \xi}{\partial\nu}=0$.
			By adding the identities \eqref{fs} and \eqref{fs1}, and using Young's inequality, we find
			\begin{eqnarray}
				\lefteqn{\int_0^t \int_{\Omega} |u_n''(\tau)|^2 dx d\tau+ r\int_0^t\int_{\Omega} |\nabla u_n''(\tau)|^2 dx d\tau+ \int_0^t\int_{\Omega} |\theta_n'(\tau)|^2 dx d\tau+\frac{1}{2}\int_{\Omega} |\nabla \theta_n(t)|^2 dx} \nonumber\\
				& \leq & \frac{1}{2} \int_0^t \int_{\Omega} |F(x,\tau)|^2 dx d\tau + \frac{1}{2}\int_0^t \int_{\Omega} |u_n''(\tau)|^2 dx d\tau+\int_0^t \int_{\Omega} |\Delta u_n'(\tau)|^2 dx d\tau+\frac{1}{2}\int_{\Omega} |\Delta u_{0,n}|^2 dx\nonumber\\
				&&  + \frac{1}{2}\int_{\Omega} |\Delta v_{0,n}|^2dx +\frac{1}{2}\int_{\Omega}|\Delta u_n(t)|^2 dx +\frac{1}{2}\int_{\Omega} |\Delta u_n'(t)|^2 dx  +\frac{1}{2} \int_0^t \int_{\Omega} |G(x,\tau)|^2 dx d\tau \nonumber \\ 
				&& + \frac{1}{2} \int_0^t \int_{\Omega} |\theta_n'(\tau)|^2 dx d\tau+\frac{1}{2}
				\Vert \alpha\Vert^2_{L^{\infty}(\Omega)} \int_{\Omega} |\nabla u_n'(t)|^2 dx + \frac{1}{2}\int_{\Omega} |\nabla \theta_n(t)|^2 dx \nonumber \\
				&& +\frac{1}{2}
				\Vert \alpha\Vert^2_{L^{\infty}(\Omega)} \int_{\Omega}|\nabla v_{0,n}|^2 dx+ \frac{1}{2}\int_{\Omega}|\nabla \theta_{0,n}|^2 dx.  \label{es}
			\end{eqnarray}
			Finally, using the inequality \eqref{S-EI},   and the estimates  \eqref{mvc1}, \eqref{7.13}, \eqref{S-RW1} and  \eqref{cz} in \eqref{es}, we get
				\begin{eqnarray*}
				\lefteqn{\frac{1}{2}\int_0^t \int_{\Omega} |u_n''(\tau)|^2 dx d\tau+ r\int_0^t\int_{\Omega} |\nabla u_n''(\tau)|^2 dx d\tau+ \frac{1}{2}\int_0^t\int_{\Omega} |\theta_n'(\tau)|^2 dx d\tau} \nonumber\\
				&&  \leq \frac{1}{2} \int_0^t \int_{\Omega} |F(x,\tau)|^2 dx d\tau +\frac{1}{2} \int_0^t \int_{\Omega} |G(x,\tau)|^2 dx d\tau\nonumber\\
				&&  +\frac{1}{2}\int_{\Omega} |\Delta u_{0,n}|^2 dx+ \frac{1}{2}\int_{\Omega} |\Delta v_{0,n}|^2dx +\frac{1}{2}\Vert \alpha\Vert^2_{L^{\infty}(\Omega)} \int_{\Omega}|\nabla v_{0,n}|^2 dx+ \frac{1}{2}\int_{\Omega}|\nabla \theta_{0,n}|^2 dx \nonumber\\
				&&+\underbrace{\left(T+\frac{1}{2}\right) \frac{1}{r} K(\alpha,r,\widetilde{C},T) \ R_2(\theta_0,u_0,v_0,F,G,\alpha)}_{:=K_1} +\underbrace{\frac{1}{2} \left(1+\frac{1}{r}\widetilde{C}\,
				\Vert \alpha\Vert^2_{H^2(\Omega)} \right) \exp(T)  R_1 (\theta_0, u_0, v_0, F,G)}_{:=K_2}.
			\end{eqnarray*}
			Next, by implementing the following projection inequalities: 
            \begin{eqnarray*}
            \|\nabla v_{0,n}\|^2_{L^2(\Omega)} &\leq& \|\nabla v_0\|^2_{L^2(\Omega)}, \|\Delta  v_{0,n}\|^2_{L^2(\Omega)}\leq \|\Delta v_0\|^2_{L^2(\Omega)}, \\
            \|\nabla\Delta u_{0,n}\|^2_{L^2(\Omega)}&\leq& \|\nabla \Delta u_0\|^2_{L^2(\Omega)} \ \ \
            \text{and} \ \ \|\nabla \theta_{0,n}\|^2_{L^2(\Omega)}\leq \|\nabla \theta_0\|^2_{L^2(\Omega)},    
            \end{eqnarray*}
            and using the weak sequential lower semi-continuity, we find that 
			there exist subsequences of $u_n''$ and $\theta_n'$ that converge weakly to the functions
			$u_{tt}\in L^{2}(0,T; H_0^1(\Omega)), \ \theta_t \in L^2(0,T;L^2(\Omega)),$
			respectively. Moreover, the following estimate holds:
			\begin{align*}\label{lk}
				\int_0^T \int_{\Omega} |u''(\tau)|^2 dx d\tau&+r \int_0^T\int_{\Omega} |\nabla u''(\tau)|^2 dx d\tau +\int_0^T\int_{\Omega} \theta'(\tau)^2 dx d\tau \nonumber\\
                &\leq R_3(F,G,\theta_0,u_0,v_0,\alpha) +K_1+K_2,
	\end{align*}
			where 
			\begin{align*}
				R_3(F,G,\theta_0,u_0,v_0,\alpha):= & \frac{1}{2} \int_0^T \int_{\Omega} |F(x,\tau)|^2 dx d\tau +\frac{1}{2} \int_0^T \int_{\Omega} |G(x,\tau)|^2 dx d\tau  +\frac{1}{2}\int_{\Omega} |\Delta u_0|^2 dx\nonumber\\
				&  + \frac{1}{2}\int_{\Omega} |\Delta v_0|^2dx+\frac{1}{2}
				\Vert \alpha\Vert^2_{L^{\infty}(\Omega)} \int_{\Omega}|\nabla v_0|^2 dx+ \frac{1}{2}\int_{\Omega}|\nabla \theta_0|^2 dx.
			\end{align*}
			This completes the proof.
		\end{proof}

\section{The inverse  coefficient problem}\label{S-IP}
		This section addresses the inverse coefficient problem, which constitutes the central objective of the present study. We begin by formulating the identification of the unknown thermal expansion coefficient as a constrained optimization problem using the Tikhonov regularization framework. Subsequently, we analyze the associated input-output operator and establish its compactness and Lipschitz continuity properties. Finally, we prove the existence of a minimizer for the regularized objective functional. 	
		
		For a given constant $\kappa>0$, consider the following admissible set of coefficients: 
		$$\mathcal{M}=\{ \alpha \in H^1(\Omega):\Vert \alpha \Vert_{H^1(\Omega)}\leq \kappa\}.$$ 
		We define the input-output operator as follows:
		\begin{eqnarray*}
			\Phi: \mathcal{M}\subset H^1(\Omega)\mapsto H_0^1(\Omega)\subset L^2(\Omega),
			\ \ \Phi(\alpha)(x):=u(x,t;\alpha)|_{t=T}.
		\end{eqnarray*}
		We shall write the inverse problem in terms of functional equation 
		\begin{eqnarray}\label{l1}
			\Phi \alpha=u_T, \alpha \in \mathcal{M}, \ u_T \in L^2(\Omega).
		\end{eqnarray} 
		Since the measured data $ u_T(x) $ contain measurement errors, achieving exact equality in the functional equation \eqref{l1} is not feasible in practice. Therefore, we reformulate the inverse problem as a minimization problem using the Tikhonov functional \begin{eqnarray*}
			J(\alpha)=\frac{1}{2}\Vert \Phi(\alpha)-u_T\Vert^2_{L^2(\Omega)}.
		\end{eqnarray*}
		The regularized form of Tikhonov functional is 
		\begin{eqnarray}\label{l12}
			J_{\lambda}(\alpha)=\frac{1}{2}\Vert \Phi(\alpha)-u_T\Vert^2_{L^2(\Omega)}+\frac{\lambda}{2}\Vert \nabla \alpha\Vert^2_{L^2(\Omega)}.
		\end{eqnarray}
		where $\lambda>0$ is the parameter of regularization, which is a constant. We now formulate the inverse problem as a minimization problem for the functional $J_{\lambda}(\alpha)$ over the set $ \mathcal{M}.$ Specifically, our goal is to solve the following:
        
		\textbf{Minimization Problem: } Find $ \alpha \in \mathcal{M}$ that minimizes the functional $ J_{\lambda}(\alpha),$ that is, \begin{eqnarray*}\label{m1}
			\min_{\alpha \in \mathcal{M}} J_{\lambda}(\alpha),
		\end{eqnarray*}
		subject to the condition that $ u(x,t;\alpha)$ solves the problem \eqref{1a}-\eqref{1d}.
		\begin{remark}
			The nonlinear nature of the inverse coefficient problem of recovering $\alpha$ in \eqref{1a}-\eqref{1d} necessitate the choice of more regularized functional \eqref{l12} with regularization term $\frac{\lambda}{2}\Vert \nabla \alpha\Vert^2_{L^2(\Omega)}$ instead of usual $ L^2$ norm regularizer $ \frac{\lambda}{2} \Vert  \alpha\Vert^2_{L^2(\Omega)}$ for the existence of solutions for the inverse problem. Furthermore, since the inverse coefficient problem of determining $\alpha$ via the  minimization problem defined above is non-convex, the uniqueness of the solution to the inverse problem cannot, in general, be guaranteed. 
		\end{remark}
		\subsection{Compactness and Lipschitz continuity of the input output operator}
		In this subsection, we establish the compactness and Lipschitz continuity of the input–output operator associated with the inverse problem. The compactness result follows from the a priori estimates obtained for weak solutions of the direct problem. In contrast, the nonlinear dependence of the state variables on the unknown coefficient introduces additional analytical difficulties in proving Lipschitz continuity of the input–output mapping.
		
		\begin{proposition}
			Suppose  Assumption \ref{a2} holds. Then the input-output operator $ \Phi: \mathcal{M}\subset H^1(\Omega)\mapsto  L^2(\Omega)$
			defined by $\Phi(\alpha)(x):=u(x,t;\alpha)|_{t=T}$ is compact.    
		\end{proposition}
		\begin{proof}
			Let the input output operator $ \Phi$ map the bounded sequence $ \{\alpha_m\}$ in $ H^1(\Omega)$ to the output sequence  $u(x,t;\alpha_m)$ at $t=T$ denoted by $ u_{Tm}.$ To prove the operator's compactness, we must show that the output sequence $ u_{Tm}$ is bounded in $ H_0^1(\Omega). $ Using the priori estimates \eqref{ke}, \eqref{5.4}, and Remark \ref{r1}, we get $ u_{Tm}\in C([0,T];H_0^1(\Omega)). $ This implies that $ u_{Tm}$ is bounded in $ H_0^1(\Omega). $ Since $ H_0^1(\Omega)$ is compactly embedded in $ L^2(\Omega),$ the sequence $ u_{Tm}$ is precompact in $ L^2(\Omega),$ that is, the operator $\Phi$ maps a bounded sequence to a precompact sequence. Hence $ \Phi$ is compact.
		\end{proof}
        Next, we demonstrate that the input-output operator $\Phi$ is Lipschitz continuous. As noted earlier, the nonlinear nature of the inverse problem significantly complicates this analysis. In the case of a linear inverse problem, such as source or boundary identification, we can directly apply the a priori estimates of the solution to the direct problem. However, in our current nonlinear setting, such an approach is insufficient. To proceed, we introduce a new admissible set of coefficients, denoted by $$\mathcal{M}_1=\{ \alpha \in H^2(\Omega):  \ \Vert \alpha \Vert_{H^2(\Omega)}\leq \kappa_1,\ \kappa_1>0\},$$ and define the input-output operator $\Phi$ on this set. Within this framework, we  establish independent estimates for the solution, $(\delta u(x,t),\delta \theta(x,t))=(u(x,t;\alpha_1)-u(x,t;\alpha_2),\theta(x,t;\alpha_1)-\theta(x,t;\alpha_2))$ of the following problem

	\begin{subnumcases}{\label{2.1}}
		\delta u_{tt} - r\,\Delta \delta u_{tt} + \Delta^{2}\delta u
		+ \operatorname{div}\!\big(\alpha_{1}(x)\nabla \delta\theta\big)
		= -\operatorname{div}\!\big(\delta\alpha(x)\nabla \theta(x,t;\alpha_{2})\big),
		\quad (x,t)\in \Omega_T, \label{2.1a} \\		
		\delta\theta_{t} - \Delta\delta\theta
		- \operatorname{div}\!\big(\alpha_{1}(x)\nabla \delta u_{t}\big)
		= \operatorname{div}\!\big(\delta\alpha(x)\nabla u_{t}(x,t;\alpha_{2})\big),
		\quad \quad \quad \quad \qquad \ \, (x,t)\in \Omega_T, \label{2.1b} \\	
		\delta u = \frac{\partial \delta u}{\partial\nu} = 0,\quad
		\delta\theta = 0,
		\qquad\qquad\qquad\quad\qquad\qquad \qquad \qquad \quad \,(x,t)\in \Gamma_T, \nonumber \\		
		\delta u(x,0) = 0,\quad
		\delta u_{t}(x,0) = 0,\quad
		\delta\theta(x,0) = 0,
		\qquad\qquad\qquad\qquad\quad\ \quad  x\in \Omega \nonumber.
	\end{subnumcases}	
%		\begin{subnumcases}{}\label{2.1}
%			&		$ \delta u_{tt}-r\Delta \delta u_{tt}+\Delta ^2 \delta u+ \mbox{div}(\alpha_1(x)\nabla \delta \theta)$ \nonumber \\&$ \; \qquad \qquad \qquad = -\mbox{div}(\delta\alpha(x)\nabla \theta(x,t;\alpha_2)), \;\quad \quad \quad \quad\qquad \quad (x,t)\in \Omega_T,$   \label{2.1a} \\ [1pt]
%			&$\delta\theta_{t}- \Delta\delta\theta -\mbox{div}( \alpha_1(x) \nabla \delta u_t)= \mbox{div}(\delta \alpha(x) \nabla u_t(x,t;\alpha_2)),\ \ \ \ (x,t)\in \Omega_T,$
%			\label{2.1b} \\ [1pt]
%			&$ \delta u=\frac{\partial \delta u}{\partial\nu}=0, \ \delta\theta =0 \qquad \qquad\; \qquad \qquad \qquad \qquad \qquad \qquad (x,t)\in \Gamma_T,$ \nonumber  \\
%			&	 $	\delta u(x,0) = 0, \delta u_{t}(x,0)  = 0, \; \delta\theta(x,0)= 0 \qquad \ \qquad \qquad \qquad \qquad x  \in \Omega,$ \nonumber
%		\end{subnumcases}
		where $\delta\alpha(x)=\alpha_1(x)-\alpha_2(x),$ $ (u(x,t;\alpha_1),\theta(x,t;\alpha_1))$ and $(u(x,t;\alpha_2),\theta(x,t;\alpha_2))$ are the solutions of the direct problem \eqref{1a}-\eqref{1d} corresponding to the thermal expansion coefficients $\alpha_1,$ $\alpha_2\in \mathcal{M}_1 \subset \mathcal{M}$ respectively, for common initial and boundary data.  
		
		\begin{proposition}
			Suppose Assumption \ref{a2} holds true. Then the input-output operator $\Phi: \mathcal{M}_1\subset H^2(\Omega)\to L^2(\Omega)$ defined by $\Phi(\alpha)(x):=u(x,t;\alpha)|_{t=T}$ is Lipschitz continuous, that is
			\begin{eqnarray*}
				\Vert \Phi(\alpha_1)-\Phi(\alpha_2)\Vert_{L^2(\Omega)}\leq T \, \sqrt{L_0} \, \Vert \alpha_1-\alpha_2 \Vert _{H^2(\Omega)}, \ \forall \ \alpha_1,\alpha_2 \in \mathcal{M}_1,
			\end{eqnarray*} 
			where the Lipschitz constant $ L_0=\exp(T/r) \left( \frac{T}{r}+ \frac{1}{2}\right)R_1(\theta_0, u_0, v_0, F,G)\   \widetilde{C}$,
			in which the constants $\widetilde{C}$ and $R_1$ are coming from \eqref{S-EI} and \eqref{r1q}, respectively.

		\end{proposition}
		\begin{proof}
			From the equality $\delta u(T)= \delta u(0)+\int_0^T\delta u_t(\tau)d\tau$, we can establish the inequality $ \Vert \delta u(\cdot,T)\Vert_{L^2(\Omega)}\leq \sqrt{T} \ \Vert \delta u_t\Vert_{L^2(0,T;L^2(\Omega))}$. Now, from the definition of the operator $\Phi$ and this inequality, we can obtain
			\begin{eqnarray}\label{l1234}
				\Vert \Phi(\alpha_1)-\Phi(\alpha_2)\Vert^2_{L^2(\Omega)}=\Vert \delta u(\cdot,T)\Vert^2_{L^2(\Omega)}\leq T \Vert \delta u_t\Vert^2_{L^2(0,T;L^2(\Omega))}.
			\end{eqnarray} 
			To obtain the estimate for $\Vert \delta u_t\Vert^2_{L^2(0,T;L^2(\Omega))},$ we multiply  \eqref{2.1a} by $ \delta u_t$ and \eqref{2.1b} by $\delta \theta,$  perform integration by parts over $ \Omega \times [0,t],$ to obtain
			\begin{eqnarray}
				\lefteqn{\frac{1}{2}\int_{\Omega} |\delta u_t(t)|^2 dx+\frac{r}{2}\int_{\Omega}|\nabla \delta u_t(t)|^2 dx+\frac{1}{2}\int_{\Omega} |\Delta \delta u(t)|^2 dx} \nonumber \\ [1pt]&&-\int_0^t \int_{\Omega} \alpha_1(x)\nabla \delta\theta\cdot\nabla \delta u_\tau dx d\tau=\int_0^t \int_{\Omega} \delta \alpha(x) \nabla \theta(x,t;\alpha_2) \cdot\nabla \delta u_{\tau} \ dx d\tau. \label{p21} \\ [1pt]
				\lefteqn{\frac{1}{2} \int_{\Omega} \delta \theta(t) ^2 dx+ \int_0^t \int_{\Omega}|\nabla \delta \theta |^2 dx d\tau+\int_0^t \int_{\Omega} \alpha_1(x)\nabla \delta u_{\tau}\cdot \nabla \delta \theta \ dx d\tau   }\nonumber \\ [1pt] &&= -\int_0^t \int_{\Omega} \delta \alpha(x) \nabla u_{\tau}(x,\tau;\alpha_2)\cdot \nabla \delta \theta \ dx d\tau. \label{p21.1}
			\end{eqnarray}
			Adding \eqref{p21}, \eqref{p21.1} and applying Cauchy's inequality, we get
			\begin{eqnarray*}
				\lefteqn{\int_{\Omega} \left|\delta u_t(t)\right|^2 dx+r\int_{\Omega}|\nabla \delta u_t(t)|^2 dx+\int_{\Omega} |\Delta \delta u(t)|^2 dx+\int_{\Omega} |\delta \theta(t)|^2 dx}\nonumber \\ [1pt]&&+2\int_0^t\int_{\Omega}|\nabla \delta \theta|^2 dx d\tau \leq \Vert \delta \alpha \Vert^2_{L^{\infty}(\Omega)}\left(\int_0^t \int_{\Omega}| \nabla u_{\tau}(x,\tau;\alpha_2)|^2\ dx d\tau+\int_0^t \int_{\Omega} |\nabla \theta(x,\tau;\alpha_2)|^2 dx d\tau\right)\nonumber \\ [1pt]&&+ \int_0^t \int_{\Omega} |\nabla \delta\theta|^2 dx d\tau  + \int_0^t \int_{\Omega} | \nabla \delta u_{\tau}|^2 dx d\tau.\label{582}
			\end{eqnarray*}
			By invoking Gronwall's inequality, estimates \eqref{mvc1}, \eqref{7.11}, and the inequality \eqref{S-EI}, we derive
				\begin{eqnarray}
				\lefteqn{\int_{\Omega} \left|\delta u_t(t)\right|^2 dx+r\int_{\Omega}|\nabla \delta u_t(t)|^2 dx+\int_{\Omega} |\Delta \delta u(t)|^2 dx+\int_{\Omega} |\delta \theta(t)|^2 dx}\nonumber \\ [1pt]&&\hspace{1cm}+\int_0^t\int_{\Omega}|\nabla \delta \theta|^2 dx d\tau \leq \exp(T/r) \left( \frac{T}{r}+ \frac{1}{2}\right)R_1(\theta_0, u_0, v_0, F,G)\   \widetilde{C}\,\Vert \delta \alpha  \Vert^2_{H^2(\Omega)}.\label{581}
			\end{eqnarray}
			%\begin{eqnarray}
			%	\int_{\Omega}|\nabla \delta u_t(t)|^2 dx \leq \exp(T/r) \left( \frac{T}{r}+ \frac{1}{2}\right)R_1 \Vert \delta \alpha  \Vert^2_{L^{\infty}(\Omega)}.\label{581}
			%end{eqnarray}
			%Using \eqref{581} in \eqref{582} and applying the embedding $ H^2(\Omega)\hookrightarrow L^{\infty}(\Omega),$ we further get 
			%\begin{eqnarray}
			%	\int_0^T \int_{\Omega} \delta u_t^2 dx dt &\leq& TC \exp(T/r)\left(C_e+1\right)R_1 \left(1+\frac{T}{r}\right)\Vert \delta \alpha  \Vert^2_{H^2(\Omega)},\label{l123} \\
			%	\int_0^T \int_{\Omega} |\nabla\delta \theta|^2 dx dt &\leq& C\exp(T/r)\left(C_e+1\right)R_1 \left(1+\frac{T}{r}\right)\Vert \delta \alpha  \Vert^2_{H^2(\Omega)}.\label{l1231}
			%\end{eqnarray}
			By substituting \eqref{581} into \eqref{l1234}, we deduce the Lipschitz continuity of the input-output operator.
		\end{proof}
\begin{remark}
	The Lipschitz continuity of the input-output operator $ \Phi$ shows that the forward problem is stable with respect to the perturbation of the coefficient $\alpha$,	while	the compactness of this operator leads to the ill-posedness of the coefficient inverse problem \eqref{1a}-\eqref{1d} and \eqref{i1} (see, \cite{hadamard1964}, \cite{hasanouglu2017introduction}).
		\end{remark}	
Finally, we prove the main result of this paper, showing the existence of solutions to the inverse coefficient problem \eqref{1a}-\eqref{1d} and \eqref{i1}.
\begin{theorem}
			Suppose $u_T \in L^2(\Omega)$ and the assumptions of Theorem \ref{t1} hold. Then there exists an admissible coefficient $ \alpha^{\ast}\in \mathcal{M}$ that minimizes the regularized Tikhonov functional $J_{\lambda}(\alpha)$ defined by \eqref{l12}.   
		\end{theorem} 
		\begin{proof}
			Since $\lambda$ is positive, from the definition of the regularized functional \eqref{l12}, the set $\big\{J_{\lambda}(\alpha) \ :\  \alpha \in \mathcal{M} \big\}$ is bounded below by zero. Then, there exists a minimizing sequence $ \{ \alpha_k\}\in \mathcal{M}$ such that
			\begin{eqnarray}\label{mnbvc}
				\lim_{k\mapsto \infty } J_{\lambda}(\alpha_k)=\inf_{\alpha \in \mathcal{M}} J_{\lambda}(\alpha).
			\end{eqnarray}
			Moreover, from the definition of the set $\mathcal{M}$, we have the uniform bound $\Vert \alpha_k \Vert_{H^1(\Omega)}\leq \kappa,$ for all $ k$. Then there exists a bounded subsequence which is again denoted by $ \{ \alpha_k\},$  such that $\alpha_k \rightharpoonup \alpha^{\ast}$ weakly in $ H^1(\Omega).$ Since $ \mathcal{M}$ is a closed convex subset of the Hilbert Space $H^1(\Omega)$, therefore, the weak limit $ \alpha^{\ast} \in \mathcal{M},$ and thus $\alpha^{\ast}$ is an admissible coefficient. 
			
			The estimate \eqref{5.4} indicates that the sequence $ \left(u(x,t;\alpha_k), \theta(x,t;\alpha_k)\right)$ denoted by $ \left(u(\alpha_k), \theta(\alpha_k)\right)$ are bounded in $ L^2(0,T;H_0^2(\Omega))\times L^2(0,T;H_0^1(\Omega))$. Consequently, there exist subsequences $\{u_k\}$ of $u(\alpha_k)$ and $ \{\theta_k\}$ of $ \theta(\alpha_k)$ such that $ u_k \rightharpoonup u^{\ast}$ in $ L^2(0,T;H_0^2(\Omega))$ and $\theta_k \rightharpoonup \theta^{\ast}$ in $L^2(0,T;H_0^1(\Omega)).$ To show that $ u^{\ast}(x,t)=u(x,t; \alpha^\ast)$ and $ \theta^*(x,t)=\theta(x,t;\alpha^\ast),$ we choose test functions $ v \in C^2([0,T];H_0^2(\Omega))$ and $ w\in C^1([0,T];H_0^1(\Omega))$ with $ v(T)=0,$ $ v'(T)=0,$ $ w(T)=0$ in the weak form \eqref{8.111a}-\eqref{8.111b}. Subsequently applying integration by parts, we arrive at
			\iffalse             
            \begin{eqnarray}
				\lefteqn{\int_0^T \int_{\Omega} u_k(t) v^{\prime \prime}(t)dxdt+ \int_0^T\int_{\Omega} r\nabla u_{k}(t)\cdot\nabla v^{\prime \prime}(t) dx dt + \int_0^T \int_{\Omega} \Delta u_k(t)\Delta v(t) dx dt}\nonumber \\ &&-\int_0^T \int_{\Omega} \alpha_{k}(x) \big( \nabla \theta_k(t)\cdot \nabla v(t) \big) dx dt= \int_0^T  \int_{\Omega} F(t) v(t) dx dt- \int_{\Omega} u_k(0)v'(0) dx \nonumber \\ && + \int_{\Omega} u_k'(0) v(0)dx +\int_{\Omega} r \Delta u_k(0)v'(0) dx - \int_{\Omega} r \Delta u_k'(0) v(0) dx. \label{nbv}
			\end{eqnarray}     
			and
			\begin{eqnarray}
				\lefteqn{\hspace{-1.5in}-\int_0^T \int_{\Omega} \theta_k(t)w'(t) dx dt + \int_0^T \int_{\Omega} \nabla \theta_k(t).\nabla w(t) dx dt+ \int_0^T \int_{\Omega} \alpha_k(x) \big( \nabla u_k'(t)  \cdot \nabla w(t) \big) dx dt} \nonumber \\ &&= \int_0^T \int_{\Omega} G(t)w(t) dx dt + \int_{\Omega}\theta_k(0) w(0)dx.\label{nbva}
			\end{eqnarray} 
			Now rearranging the terms of \eqref{nbv} and \eqref{nbva}, we get 
            \fi
			\begin{eqnarray}
				\lefteqn{\int_0^T \int_{\Omega} u_k(t) v^{\prime \prime}(t)dxdt+ r\int_0^T\int_{\Omega} \nabla u_{k}(t)\cdot \nabla v^{\prime \prime}(t) dx dt + \int_0^T \int_{\Omega} \Delta u_k(t)\Delta v(t) dx dt}\nonumber \\ && - \int_0^T \int_{\Omega} \alpha^{\ast}(x) \nabla \theta_{k}(t) \cdot \nabla v(t) dx dt -\int_0^T \int_{\Omega} \left(\alpha_{k}(x)-\alpha^{\ast}(x)\right) \nabla \theta_k(t)\cdot\nabla v(t) dx dt \nonumber \\&& = \int_0^T F(t) v(t) dx dt- \int_{\Omega} u_k(0)v'(0) dx  + \int_{\Omega} u_k'(0) v(0)dx \nonumber \\&& \qquad+\int_{\Omega} r\Delta u_k(0)v'(0) dx  - \int_{\Omega} r \Delta u_k'(0) v(0) dx,\label{bvc}
			\end{eqnarray}
			and
			\begin{eqnarray}
				\lefteqn{-\int_0^T \int_{\Omega} \theta_k(t)w'(t) dx dt + \int_0^T \int_{\Omega} \nabla \theta_k(t)\cdot\nabla w(t) dx dt } \nonumber \\ && +\int_0^T \int_{\Omega} \alpha^{\ast}(x) \big( \nabla u_k'(t)\cdot \nabla w(t)\big)  dx dt+ \int_0^T \int_{\Omega} \left(\alpha_k(x)-\alpha^{\ast}(x)\right) \nabla u_k'(t) \cdot  \nabla w(t) dx dt \nonumber \\ &&= \int_0^T \int_{\Omega} G(t)w(t) dx dt + \int_{\Omega}\theta_k(0) w(0)dx.\label{bvc1}
			\end{eqnarray}
			We only show the convergences  of integrals involving $\{\alpha_k\}.$ First, we consider the following integral $\int_0^T \int_{\Omega} \left(\alpha_{k}(x)-\alpha^{\ast}(x)\right) \nabla \theta_k(t)\cdot\nabla v(t) dx dt$ appearing in \eqref{bvc}, apply H\"older's inequality and the continuous embedding $H^1(\Omega)\hookrightarrow L^4(\Omega)$ to obtain
			\begin{eqnarray} \label{5.41}
				\lefteqn{\hspace{-.5in}\int_0^T \int_{\Omega} \left(\alpha_{k}(x)-\alpha^{\ast}(x)\right) \nabla \theta_k(t)\cdot\nabla v(t) dx dt \leq  \Vert \alpha_k-\alpha^{\ast}\Vert_{L^4(\Omega)}  \Vert \nabla v \Vert_{L^2(0,T;L^4(\Omega))} \Vert \nabla \theta_k \Vert_{L^2(0,T;L^2(\Omega))}}\nonumber \\ 
                &\leq & C \Vert \alpha_k-\alpha^{\ast}\Vert_{L^2(\Omega)}^{1/2} \Vert\nabla\alpha_k-\nabla\alpha^{\ast}\Vert_{L^2(\Omega)}^{1/2}\Vert  v \Vert_{L^2(0,T;H_0^2(\Omega))} \Vert \nabla \theta_k \Vert_{L^2(0,T;L^2(\Omega))},
			\end{eqnarray}
	where we used the 2D Ladyzhenskaya's inequality $\|u\|_{L^4(\Omega)} \leq C\|u\|_{L^2(\Omega)}^{1/2} \|\nabla u\|_{L^2(\Omega)}^{1/2}.$
			Since $ \alpha_k\rightharpoonup \alpha^{\ast} $  weakly in $H^1(\Omega)$, we can find a subsequence (again represented as $\{\alpha_k\}$), such that $\Vert \alpha_k-\alpha^{\ast}\Vert_{L^2(\Omega)} \rightarrow 0$. Combining this with inequality \eqref{5.41} implies that $\int_0^T \int_{\Omega} \left(\alpha_{k}(x)-\alpha^{\ast}(x)\right) \nabla \theta_k(t)\cdot\nabla v(t) dx dt \rightarrow 0$ as $ k\rightarrow \infty.$ 
			
            Next doing time integration by parts for the following integral occuring in \eqref{bvc1}, we get 
            \begin{eqnarray}
               \lefteqn{\int_0^T \int_{\Omega} \big(\alpha_k(x)-\alpha^{\ast}(x)\big) \nabla u_k(t)\cdot  \nabla w'(t) dx dt}\nonumber\\
               &&= -\int_0^T \int_{\Omega} \left(\alpha_k(x)-\alpha^{\ast}(x)\right) \nabla u_k(t)\cdot  \nabla w^{\prime}(t) dx dt +\int_{\Omega} \left(\alpha_k(x)-\alpha^{\ast}(x)\right) \nabla u_0(x)\cdot  \nabla w(x,0) dx\nonumber\\
               &&:= I_1+I_2.  
            \end{eqnarray}
   By following the arguments similar to \eqref{5.41}, we get        
			\begin{eqnarray}\label{234}
			I_1&\leq &\bigg(\int_0^T \int_{\Omega}\left(\alpha_{k}(x)-\alpha^{\ast}(x)\right)^2|\nabla u_k(t)|^2 dx dt\bigg)^{\frac{1}{2}} \bigg(\int_{0}^T\int_{\Omega} |\nabla w'(t)|^2 dxdt\bigg)^{\frac{1}{2}} \nonumber \\ & \leq&\Vert \alpha_k-\alpha^{\ast}\Vert_{L^2(\Omega)}^{1/2} \Vert\nabla\alpha_k-\nabla\alpha^{\ast}\Vert_{L^2(\Omega)}^{1/2}	\Vert u_k\Vert_{L^2(0,T;H_0^2(\Omega))} \Vert \nabla w'\Vert_{L^2(0,T;L^2(\Omega))} 
                \end{eqnarray}
   and             
                \begin{eqnarray}\label{2341}
				\lefteqn{\int_{\Omega} \left(\alpha_k(x)-\alpha^{\ast}(x)\right) \nabla u_0(x)\cdot  \nabla w(x,0) dx }\nonumber \\ &\leq&
				\bigg(\int_{\Omega}\left(\alpha_{k}(x)-\alpha^{\ast}(x)\right)^2|\nabla u_0(x)|^2 dx \bigg)^{\frac{1}{2}} \bigg(\int_{\Omega} |\nabla w(x,0)|^2 dx\bigg)^{\frac{1}{2}} \nonumber \\ & \leq&\Vert \alpha_k-\alpha^{\ast}\Vert_{L^2(\Omega)}^{1/2} \Vert\nabla\alpha_k-\nabla\alpha^{\ast}\Vert_{L^2(\Omega)}^{1/2}	\Vert u_0\Vert_{H_0^2(\Omega)} \Vert \nabla w(0)\Vert_{L^2(\Omega)} 
                \end{eqnarray}
			The inequalities \eqref{234}, \eqref{2341} and $\alpha_k\to  \alpha^{\ast} $ in $ L^2(\Omega)$  implies that the following convergence holds:  $\int_0^T \int_{\Omega} \left(\alpha_k(x)-\alpha^{\ast}(x)\right) \nabla u_k'(t)\cdot\nabla w(t) dx dt \rightarrow 0$ as $ k \rightarrow \infty.$
			
			By taking the limit $ k\rightarrow \infty $ in \eqref{bvc}, \eqref{bvc1} and doing integration by parts, it can be verified that $ u^{\ast}(x,t)=u(x,t,\alpha^{\ast})$ and $\theta^{\ast}(x,t)=\theta(x,t,\alpha^{\ast})$ form a weak solution of the system \eqref{1a}-\eqref{1d}. 
            
			Next we want to show that $ \inf_{\alpha \in \mathcal{M}} J_{\lambda}(\alpha)=J_{\lambda}(\alpha^{\ast}),$ that is, $ \alpha^{\ast}$ is the optimal solution for the functional $ J_{\lambda}(\alpha).$ It is clear that 
			\begin{eqnarray*}
				\int_{\Omega} \big|\big (u(x,T;\alpha_k) -u_{T}(x)\big)- \big ( u(x,T;\alpha^{\ast})-u_{T}(x)\big)\big|^2 dx\geq 0,
			\end{eqnarray*}
			whence
			\begin{eqnarray*}
				\lefteqn{\int_{\Omega} \left( u(x,T;\alpha_k)-u_{T}(x)\right)^2 dx  \geq - \int_{\Omega} \left( u(x,T;\alpha^{\ast})-u_T(x)\right)^2 dx}\nonumber \\ &+&2 \int_{\Omega} \left( u(x,T;\alpha_k)-u_T(x)\right) \left( u(x,T;\alpha^{\ast})-u_T(x)\right) dx.
			\end{eqnarray*}
			By \eqref{7.13}, $ u(x,T;\alpha_k)\rightharpoonup u(x,T;\alpha^{\ast})$ in $ H_0^2(\Omega),$ and hence the above inequality leads to
			\begin{eqnarray}
				\liminf_{k\rightarrow \infty}\int_{\Omega} \left( u(x,T;\alpha_k)-u_{T}(x)\right)^2 dx  &\geq& \int_{\Omega} \left( u(x,T;\alpha^{\ast})-u_T(x)\right)^2.\label{ier}
			\end{eqnarray}
			The weak convergence $ \alpha_k \rightharpoonup \alpha^{\ast}$ in $ H^1(\Omega)$ and the inequality \eqref{ier} give
			\begin{eqnarray}
				\liminf_{k\rightarrow \infty} J_{\lambda}(\alpha_k)&=&\liminf_{k\rightarrow \infty}\int_{\Omega} \left( u(x,T;\alpha_k)-u_T(x) \right) ^2 dx + \liminf_{k\rightarrow\infty} \frac{\lambda}{2} \int_{\Omega}| \nabla \alpha_k |^2 dx \nonumber \\ &\geq&   \int_{\Omega} \left( u(x,T;\alpha^{\ast})-u_T(x)\right)^2 dx+ \frac{\lambda}{2} \int_{\Omega}
				| \nabla \alpha^{\ast} |^2 dx  = J_{\lambda}(\alpha^{\ast}).\label{bnmv}
			\end{eqnarray} 
			Employing \eqref{mnbvc} and \eqref{bnmv}, we deduce
			\begin{eqnarray*}
				\inf_{\alpha \in \mathcal{M}} J_{\lambda}(\alpha) \leq J_{\lambda}(\alpha^{\ast})\leq \liminf_{k\rightarrow \infty} J_{\lambda}(\alpha_k)= \inf_{\alpha \in \mathcal{M}} J_{\lambda}(\alpha).
			\end{eqnarray*}
			This shows that $ \alpha^{\ast}$ is the minimizer for the functional $ J_{\lambda}(\alpha).$ This completes the proof.
		\end{proof}
		\section{Conclusion}
        In this paper, we studied a coefficient inverse problem of reconstructing the thermal expansion coefficient in the thermoelastic plate equation from final-time displacement measurements. We first established the well-posedness of the direct problem and, under additional regularity on the unknown coefficient, derived the regularity of its solution required for the subsequent analysis. To account for the inherent noise in the measured data, the inverse problem was formulated as the minimization of a Tikhonov regularized cost functional. Using a priori bounds derived from the direct problem, we proved the compactness of the input–output operator, confirming the ill-posedness of the inverse problem in the classical sense. Within this variational framework, we established the existence of a minimizer of the regularized functional via the direct method of calculus of variations, thereby providing a rigorous solution to the inverse problem.  
        
      Owing to the nonlinear nature of the problem, the analysis of the Fréchet derivative of the Tikhonov functional via the adjoint method is considerably more technical. This analysis, together with the design and implementation of a gradient-based numerical reconstruction algorithm, will be addressed in a forthcoming paper. As inverse coefficient problems for thermoelastic plate models have received limited attention in the optimization literature, we expect the present framework to serve as a starting point for further theoretical and numerical developments in this direction.

		%	
		%	
		%	%The acknowledgments section should not be numbered.
		%%	\section*{Acknowledgments} \\
		%	
		%	%%%%%%%%%%%%%%%%%%%%%%%%%%%%%%%%%%%%%%%%%%%%%%%%%%%%%%
		%	%          7. REFERENCES SECTION
		%	%%%%%%%%%%%%%%%%%%%%%%%%%%%%%%%%%%%%%%%%%%%%%%%%%%%%%%
		%	
		%	%       READ THIS SECTION CAREFULLY
		%	
		%	% Each of the references below MUST be cited in your article above. Do not include references that are not cited in your article.
		%	
		%	% Follow the examples below carefully. We strongly suggest that you copy and paste your reference information directly into our examples.
		%	
		%	% List all references in alphabetical order according to the first author’s last name.
		%	
		%	% Verify each URL works correctly and can be accessed properly. Your URL links should be to reputable websites. The command line for a website link begins with: \url{ }
		%	
		%	% Do not add MR or DOI numbers to your references. AIMS production staff will add this information.
		%	
		%	% Using BibTex is not recommended but can be handled.
		%	
		%	
		%	
		%	\medskip
		%	% The information below will be filled in by AIMS production staff.
		%	Received xxxx 20xx; revised xxxx 20xx; early access xxxx 20xx.
		%	\medskip
		%	

\begin{thebibliography}{99}
			\bibitem{andrew2018}
			N. Andrew, Norris,
			\newblock{Dynamics of thermoelastic thin plates: a comparison of four theories, \textit{Journal of Thermal Stresses,} \textbf{29}(2006), 169-195.}
			\bibitem{anjuna:2021}
			D. Anjuna, K. Sakthivel and A. Hasanov,
			\newblock{Determination of a spatial load in a damped Kirchhoff-Love plate equation from final time measured
				data,}
			\newblock{\textit{Inverse Problems,} \textbf{38}(2022),} 15009 (35pp).
			
			\bibitem{anjuna}
			D. Anjuna, A. Hasanov, and K. Sakthivel,
			\newblock{Simultaneous identification of spatial load and external heat source in a thermoelastic plate from final time measured displacement, \textit{Inverse Problems and Imaging,} \textbf{18}(2024), 751-755.}
			\bibitem {OAA} O. Baysal, A. Hasanov and A. Kawano, Reconstruction of the spatial component in the source term of a vibrating elastic plate from boundary observation, \textit{Applied Mathematical Modelling}, \textbf{103} (2022), 409-420.
			\bibitem{bellassoued2010}
			M. Bellassoued and M. Yamamoto,
			\newblock{Carleman estimates and an inverse heat source problem for the thermoelasticity system, \textit{Inverse problems}, \textbf{27}(2010), 015006.}
			\bibitem{bhullar}
			S. K. Bhullar and J. L. Wegner, 
			\newblock{Some transient  thermoelastic plate problems,} \textit{ Journal of Thermal Stresses,} \textbf{32} (2009), 768-790. 
			\bibitem{sdbt2014}
			S. D'haeyer, B. T. Johansson and M. Soldi$\check{c}$ka,
			\newblock{Reconstruction of a spacewise-dependent heat source in a time-dependent heat diffusion process, \textit{IMA Journal of Applied Mathematics}, \textbf{79}(2014), 33-53.}
			\bibitem{meller2001}
			M. Eller, I. Laseika, and R. Triggiani
			\newblock{ Exact or approximate controllability of thermoelastic plate with variable thermal coefficients, \textit{Discrete and continuous dynamical systems,} \textbf{7}(2001,) 283-302.  }
			\bibitem{lcevanspartial2010}
			L. C. Evans,
			\newblock{\textit{Partial Differential Equations,}}
			\newblock{American Mathematical Soc, 2010.}
			\bibitem{gilbarg1998elptic}
			D. Gilbarg and N. S. Trudinger,
			\newblock{\textit{Elliptic Partial Differential Equations of Second Order, }}
			\newblock{Springer, 1998.}
			\bibitem{hadamard1964}
			L. Hadamard,
			\newblock{\textit{La th\'{e}orie des \'{e}quations aux d\'{e}riv\'{e}es partielles,}}
			\newblock{\'{E}ditions scientifiques, 1964.}
			
			\bibitem{PH} P. Hartman, 
			\newblock{\textit{Ordinary Differential Equations,}}
			\newblock{John Wiley \& Sons, Inc., New York-London-Sydney, 1964.}
			
			\bibitem{hasanov2007}
			A. Hasanov,
			\newblock{Simultaneous determination of source terms in a linear parabolic problem from the final overdetermination: Weak solution approach, \textit{Journal of Mathematical Analysis and Applications, } \textbf{330}(2007), 766-779.}
			\bibitem{hasanov2009identification}
			A. Hasanov,
			\newblock{Identification of an unknown source term in a vibrating cantilevered
				beam from final overdetermination,}
			\newblock{\textit{Inverse Problems,} \textbf{25}(2009), 115015}
			\bibitem{hasanouglu2019}
			A. Hasanov and O. Baysal, \newblock{Identification of a temporal load in a cantilever beam from
				measured boundary bending moment,} \newblock{\textit{Inverse Problems,} \textbf{35} (2019), 105005.}
			\bibitem{hasanouglu2017introduction}
			A. Hasanov and V.G. Romanov,
			\newblock{\textit{Introduction to Inverse Problems for Differential Equations} (2nd ed.),}
			\newblock{Springer, 2021.} 

           
			\bibitem{JElagnese}
			J.E. Lagnese,
			\newblock{\textit{Boundary Stabilization of Thin Plates,}}
			\newblock{SIAM, Philadelphia, PA, 1989.}
			\bibitem{J.E. Lagnese}
			J.E. Lagnese and J.L Lions,
			\newblock{\textit{Modelling
					Analysis and Control
					of Thin Plates,}}
			\newblock{Masson, Paris, 1988.}
			 \bibitem{NKM}N.K. Murugesan,  K. Sakthivel,  A. Hasanov and N. Barani Balan,  Inverse coefficient problem for cascade system of fourth and second order partial differential equations,  \emph{Applied Mathematics and Optimization},  89 (2024),  1-32.
			
			\bibitem{sakthivel2011}
			K. Sakthivel, S. Gnanavel, N. Barani Balan, and K. Balachandran,
			\newblock{Inverse problem for the reaction diffusion system by optimization method, \textit{Applied Mathematical Modelling}, \textbf{38}(2011), 571-579.} 

            \bibitem {TLK}  T. Sharma, L. Beilina and K. Sakthivel, Lagrangian approach for the reconstruction of source function in a parabolic equation using partial boundary measurements, \textit{Mathematics and Computers in Simulation,}  \textbf{250} (2026), 442-463.	
	
			\bibitem{esuhir}
			E. Suhir,
			\newblock{Thermal stress modeling in microelectronics and photonic structures and the application of the probablistic approach: Review and extension, \textit{The International journal of microcircuits and electronic packaging}, \textbf{23}(2000), 215-223. }
			\bibitem{timoshenko}
			S. Timoshenko and S. Woinowsky-Krieger,
			\newblock{\textit{Theory of Plates and Shells,} McGraw-Hill, New York
				(1987).}
            \bibitem{Tikhonov}    
            A. Tikhonov and V. Arsenin, \newblock{\textit{Solutions of Ill-Posed Problems,} Geology Press, Beijing, 1979.}
			\bibitem{van2019}	
			K. Van Bockstal,
			\newblock{Identification of an unknown spatial load distribution in a vibrating beam or plate from the final state, \textit{Journal of Inverse and Ill-posed Problems,} \textbf{27}(2019), 623-642.}
			\bibitem{van2015}
			K. Van Bockstal, and M.Slodi{\v{c}}ka,
			\newblock{Recovery of a space-dependent vector source in thermoelastic systems, \textit{Inverse Problems in Science and Engineering,} \textbf{23}(2015).} 
\bibitem{KW}K. Wehrheim, \textit{Uhlenbeck Compactness}, EMS Series of Lectures in Mathematics, 2004.
			\bibitem{wu1}
			B. Wu, and J. Liu,
			\newblock{Conditional stability and uniqueness for determining
				two coefficients in a hyperbolic–parabolic system, \textit{Inverse Problems,} \textbf{27}(2011), 075013. }		
			\bibitem{bwu}
			B. Wu, and J. Liu
			\newblock{Determination of an unknown source for a thermoelastic system with a memory effect, \textit{Inverse Problems,} \textbf{28}(2012),095012.}
			\bibitem{xxing}
			X. Xing, and Z. Shen
			\newblock{Thermoelastic structural dynamics analysis of a satellite with composite thin-walled boom, \textit{Acta Mechanica }, \textbf{3}(2023), 1250-1273. }
			
		\end{thebibliography}
	\end{document}